\documentclass[12pt,a4paper,reqno]{amsart}

\usepackage[a4paper,left=1.2in,right=1.2in,top=1.4in,bottom=1in]{geometry}
\usepackage{amsmath,amssymb,amsfonts,mathtools}
\usepackage{microtype}
\usepackage{hyperref}
\usepackage{enumitem}
\usepackage{upgreek}
\usepackage{bm}
\hypersetup{
  colorlinks=true,
  linkcolor=blue,
  citecolor=blue,
  urlcolor=blue
}

\newtheorem{theorem}{Theorem}[section]
\newtheorem{corollary}[theorem]{Corollary}
\newtheorem{lemma}[theorem]{Lemma}
\newtheorem*{mainquestion}{Main Question}
\newtheorem{proposition}[theorem]{Proposition}
\theoremstyle{definition}

\newtheorem{example}[theorem]{Example}
\theoremstyle{remark}
\newtheorem{remark}[theorem]{Remark}

\numberwithin{equation}{section}

\newcommand{\F}{\mathcal F^2}
\newcommand{\C}{\mathbb C}
\newcommand{\R}{\mathbb R}
\newcommand{\dd}{\,d}
\newcommand{\e}{\mathrm e}
\newcommand{\ip}[2]{\left\langle #1,#2\right\rangle}

\title[Complex Symmetric Weighted Composition Operators]
      {A Characterization of Complex Symmetric Weighted Composition Operators on the Fock Space}

\author{Yuanqi Sang}
\address{School of Mathematics, Southwestern University of Finance and Economics,
  Chengdu, 611130, China}
\email{sangyq@swufe.edu.cn}

\author{Liankuo Zhao}
\address{School of Mathematics and Computer Science, Shanxi Normal University,
  Taiyuan, 030031, China}
\email{lkzhao@sxnu.edu.cn}

\date{}
\thanks{}

\begin{document}

\begin{abstract}
We characterize all bounded complex symmetric weighted composition operators on the Fock space \(\mathcal F_{\alpha}^{2}\),
without prescribing a conjugation a priori. Our approach uses reproducing kernels and two generating functions. 
The Taylor coefficients of these functions are eigenvectors of the operator and its adjoint, respectively. 
The conjugations arising from our construction are anti-linear Gaussian integral operators.
Some of them are weighted composition conjugations, whereas others are not.
\end{abstract}

\maketitle

\section{Introduction}

For $\alpha>0$, the weighted Fock space $\mathcal F_\alpha^2$ is defined as the space of all entire functions $f$ such that
\[
\|f\|_\alpha^2
=\frac{\alpha}{\pi}\int_{\C}|f(z)|^2e^{-\alpha |z|^2}\dd A(z)
<\infty,
\]
where $dA$ denotes the area measure on $\C$.
Its inner product is
\[
  \ip{f}{g}_\alpha
  =\frac{\alpha}{\pi}\int_{\C}f(z)\overline{g(z)}e^{-\alpha |z|^2}\dd A(z).
\]
We refer to
\cite{Zhu2012} for further background on Fock spaces.

Given entire functions \(\Psi\) and \(\Phi\), the weighted composition
operator \(W_{\Psi,\Phi}^{(\alpha)}\) on \(\mathcal F_\alpha^2\) is defined by
\[
W_{\Psi,\Phi}^{(\alpha)}f=\Psi\cdot(f\circ\Phi).
\]
Throughout this paper, we assume that the weight function \(\Psi\) is not
identically zero. When \(\Psi=1\), \(W_{\Psi,\Phi}^{(\alpha)}\) reduces to
a composition operator. Composition operators on Hardy, Bergman, and other
Banach spaces of holomorphic functions have been studied extensively.
For comprehensive treatments of this subject, we refer the reader to
\cite{CowenMacCluer1995,Shapiro1993}.

When \(\alpha=1\), we write \(\F=\mathcal F_1^2\) and
$\ip{f}{g}=\ip{f}{g}_1,$ for all $\ f,g\in\F.$
In this case we write \(W_{\psi,\varphi}\) instead of
\(W_{\psi,\varphi}^{(1)}\).
A unitary dilation identifies
the problem on \(\mathcal F_\alpha^2\) with the corresponding problem on
\(\F\); the precise reduction is given in Section~2.

The boundedness problem for \(W_{\psi,\varphi}\)  on \(\F\) has a
rigid answer:
when \(\psi\not\equiv0\), the symbol \(\varphi\) must be affine.  More
precisely, by the results of Le \cite{Le2014},
\begin{align}\label{phi}
\varphi(z)=az+d,\qquad |a|\leq1,
\end{align}
with a corresponding growth condition on \(\psi\). This rigidity makes the Fock
space a natural setting for explicit operator-theoretic characterizations.

Complex symmetric operators provide a natural generalization of complex
symmetric matrices to operators on Hilbert spaces
\cite{GarciaProdanPutinar2014,GarciaPutinar2006,GarciaPutinar2007}.
A \emph{conjugation} on a complex Hilbert space $\mathbb{H}$ is a
anti-linear operator satisfying
$C^2=I$ and $\ip{C\bm{x}}{C\bm{y}}=\ip{\bm{y}}{\bm{x}}$ for all $\bm{x},\bm{y}\in\mathbb{H}$.
A bounded linear operator $T$ on $\mathbb{H}$ is called
\emph{complex symmetric} if there exists a conjugation $C$ on
$\mathbb{H}$ such that
$T=CT^*C$.

Concrete operators on analytic function spaces provide a natural setting in
which the abstract relation $T=CT^*C$ can be converted into verifiable
function-theoretic conditions.  Considerable progress has been made for
Toeplitz \cite{GuoZhu2014,JungKimKoLee2014,Noor2017,LiYangLu2020,BuChenZhu2021,ChenLeeZhao2022,BhuiaPradhanSarkar2025},
composition \cite{BourdonNoor2015,HanWang2022}, 
and weighted composition operators \cite{GarciaHammond2014,JungKimKoLee2014,HaiKhoi2016,LimKhoi2018,Hai2020,HuYangZhou2020,HaiTien2024},
but most available classifications prescribe the conjugation in advance.  The present paper addresses the
corresponding unrestricted problem for weighted composition operators on the
one-dimensional Fock space.

Garcia and Hammond \cite{GarciaHammond2014} studied complex symmetric
weighted composition operators on weighted
Hardy spaces. They established necessary conditions,
constructed nonnormal examples, characterized operators relative to a standard
conjugation, and connected complex symmetry with fixed points and Koenigs
eigenfunctions. Their results have been applied and extended in later studies
of composition operators on analytic function spaces. In addition to these
results, they posed the following general classification problem:

\medskip

\noindent
\begin{minipage}{\textwidth}
\centering
\itshape
Characterize all complex symmetric composition operators $C_\varphi$
on the classical Hardy space $H^2$ or, more generally, on weighted
Hardy spaces $H^2(\beta)$.
\end{minipage}

\medskip

On the Fock space $\F,$
Hai and Khoi \cite{HaiKhoi2016} introduced the family of weighted composition conjugations
\begin{align}\label{f:C}
(\mathcal{C}_{\gamma,\tau,\kappa}f)(z)
=
\kappa e^{\tau z}\,
\overline{
f\!\left(\overline{\gamma z+\tau}\right)
},
\qquad
f\in\mathcal{F}^{2},
\end{align}
where \(\gamma,\tau,\kappa\in\mathbb{C}\) satisfy
$
|\gamma|=1,\
\overline{\gamma}\tau+\overline{\tau}=0,\
|\kappa|^{2}e^{|\tau|^{2}}=1.
$
They established necessary and sufficient conditions for a bounded
weighted composition operator to be symmetric with respect to a
conjugation of this form. The reproducing-kernel method used in their
work was subsequently applied to weighted composition operators in
several variables \cite{HaiKhoi2016CR,HaiKhoi2018,HaiTien2024}, and to
unbounded weighted composition operators \cite{Hai2020}.
Their results concerning spectra and eigenvectors were also
used in subsequent studies. Their characterization applies to the
specified family \(\mathcal{C}_{\gamma,\tau,\kappa}\), rather than to arbitrary
conjugations on \(\mathcal{F}^{2}\).
It is also worth noting that Han and Wang proved that every bounded composition operator 
on $\F$ is complex symmetric \cite{HanWang2022}.

To the best of our knowledge, 
a complete classification of complex symmetric weighted composition operators on $\F$ remains open. 
The works \cite{HaiKhoi2016,HaiKhoi2018,HaiTien2024} characterize such operators with respect to weighted composition conjugations, 
while Han and Wang \cite{HanWang2022} consider composition operators without weights. 
It remains to determine whether these results account for all complex symmetric weighted composition operators. 
In particular, must a conjugation implementing complex symmetry be of weighted-composition type?
This leads to the following question.

\begin{mainquestion}
Which bounded weighted composition operators on \(\F\) are complex symmetric?
\end{mainquestion}

Our answer is complete.  In view of \eqref{phi}, we proceed with the following analysis.
\begin{enumerate}[
  label=\textbf{Case \arabic*.},
  leftmargin=*,
  align=left,
  itemsep=0.3em,
  topsep=0.3em,
  parsep=0pt
]
  \item If \(a=0\), then \(W_{\psi,\varphi}\) has rank one.

  \item If \(|a|=1\), then \(W_{\psi,\varphi}\) is a scalar multiple of a
  unitary weighted composition operator. 

  \item Suppose that \(0<|a|<1\). After translating the fixed point of
  \(\varphi\) to the origin and applying a rotation, the question reduces to the following form:
  \begin{equation}\label{eq:in-reduced-form}
  \begin{gathered}
    \varphi(z)=az,\qquad 0<|a|<1,\\
    \psi(z)=\exp(uz^2+vz),\qquad |u|<\frac12,\\
    b=\frac{v}{1-a},\qquad
    c=\frac{u}{1-a^2}\in\left[0,\frac12\right).
  \end{gathered}
  \end{equation}
\end{enumerate}

Answering the Main Question requires determining the eigenvalues and eigenvectors of both $W_{\psi,\varphi}$ 
and $W_{\psi,\varphi}^{*}$. If $W_{\psi,\varphi}$ is complex symmetric, 
then its point spectrum is $\{a^n:n\geq 0\}$ \cite{HaiKhoi2016}. Without assuming complex symmetry, 
Mengestie \cite{Mengestie2022} proved that
$z^n\exp(cz^2+bz)\in \ker\left(W_{\psi,\varphi}-a^{n}I\right)$. 
However, a complete characterization of the point spectrum and eigenspaces of $W_{\psi,\varphi}^{*}$ remains difficult.
In this paper, without assuming complex symmetry, we show that the
eigenvectors of $W_{\psi,\varphi}$ and $W_{\psi,\varphi}^{*}$ are given by the coefficient functions in the
Taylor expansions of the following two generating functions:
\[
Q(t,z)
=
\exp(tz+cz^2+bz)
=
\sum_{n=0}^{\infty}\frac{q_n(z)}{n!}t^n
\]
and
\[
H(t,z)
=
\exp\!\bigl\{t(z-\overline b)-ct^2\bigr\}
=
\sum_{n=0}^{\infty}\frac{h_n(z)}{n!}t^n,
\]
where $q_n(z)=z^nq(z),\ q(z)=\exp(cz^2+bz).$
Moreover, each
polynomial \(h_n\) is obtained from a Hermite polynomial by an
appropriate translation and dilation.

Our first main result is stated below and proved in
Theorem~\ref{lem:eigenvectors}.

\medskip
\noindent\textbf{Theorem A.}
Let \(W_{\psi,\varphi}\) be a bounded weighted composition operator with
symbols of the form \eqref{eq:in-reduced-form}. 
For each nonnegative integer \(n\), 
\(a^{n}\) and \(\bar a^{\,n}\) are eigenvalues of
\(W_{\psi,\varphi}\) and \(W_{\psi,\varphi}^{*}\), respectively,
and the corresponding eigenspaces are one-dimensional. More precisely,
\[
\ker\!\left(W_{\psi,\varphi}-a^{n}I\right)
   =\operatorname{span}\{q_n\},
\qquad
\ker\!\left(W_{\psi,\varphi}^{*}-\bar a^{\,n}I\right)
   =\operatorname{span}\{h_n\}.
\]

Next, we show that $\{q_n\}$ and $\{h_n\}$ form a complete 
biorthogonal system. Comparing their Gram matrices then yields our second 
main result, which is stated below and proved in 
Theorem~\ref{thm:main}.

\medskip
\noindent\textbf{Theorem B.}
Let \(W_{\psi,\varphi}\) be a bounded weighted composition operator with
symbols of the form \eqref{eq:in-reduced-form}. 
Then \(W_{\psi,\varphi}\) is complex
symmetric if and only if either
\begin{equation*}
\begin{gathered}
c=0;\\
\text{or}\quad
0<c<\frac12
\quad\text{and}\quad
\frac{(\operatorname{Re}b)^2}{1-2c}
=
\frac{(\operatorname{Im}b)^2}{1+2c}.
\end{gathered}
\end{equation*}
In either case, \(W_{\psi,\varphi}\) is complex symmetric with respect to
the conjugation \(J\) defined by
\[
(Jf)(z)
=
\frac{q(z)}{\pi\|q\|}
\int_{\mathbb C}
q(w)e^{\rho zw}\overline{f(w)}
e^{-|w|^2}\,dA(w),
\qquad f\in \F,
\]
where \(\rho\) is defined in \eqref{eq:rho}.

The above conjugation has the Gaussian integral kernel
\[
K(z,w)=\exp\bigl(cz^2+bz+cw^2+bw+\rho zw\bigr).
\]
When $b=0$, composing $J$ with the standard conjugation gives,
up to a unimodular scalar, a canonical integral operator in the precise sense of
Dong and Zhu \cite{DongZhu2024}; see Proposition~\ref{prop:canonical-relation}.

In Section~2, by ruling out the cases already known to be complex symmetric
and exploiting unitary equivalence, we reduce the Main Question under consideration
to the form given in \eqref{eq:in-reduced-form}.
Section~3 introduces two key generating functions, which are then used to
characterize the point spectrum and eigenspaces of the reduced weighted
composition operator.
Section~4 proves Theorem~B.
Section~5 transfers the results
to weighted Fock spaces via dilation, and determines the spectra of
complex symmetric weighted composition operators.

\section{Preliminaries and Reductions}

The reproducing kernel of $\mathcal F_\alpha^2$ at $w\in\C$ is
\[
  K_w^{(\alpha)}(z)=\e^{\alpha z\overline w},
  \qquad z\in\C.
\]
It satisfies the reproducing property
\[
  f(w)=\ip{f}{K_w^{(\alpha)}}_\alpha,
  \qquad f\in\mathcal F_\alpha^2.
\]
The normalized reproducing kernel is
$k_w^{(\alpha)}(z)=\e^{\alpha z\overline w-\alpha|w|^2/2}.$
The monomials
\[
  e_n^{(\alpha)}(z)=\left(\frac{\alpha^n}{n!}\right)^{1/2}z^n,
  \qquad n=0,1,2,\ldots ,
\]
form the standard orthonormal basis of \(\mathcal F_\alpha^2\).
Define the dilation operator
\[
    D_\alpha:\mathcal F_\alpha^2\to \mathcal F^2
\]
by $(D_\alpha f)(z)=f\left(\frac{z}{\sqrt{\alpha}}\right),
 z\in\mathbb C.$
Then \(D_\alpha\) is unitary, with inverse given by
\[
    (D_\alpha^{-1}g)(z)=g(\sqrt{\alpha}z),
    \qquad z\in\mathbb C.
\]

Let \(W_{\Psi,\Phi}^{(\alpha)}\) be a weighted composition operator on
\(\mathcal F_\alpha^2\).
Under the unitary operator \(D_\alpha:\mathcal F_\alpha^2\to\mathcal F^2\),
we have
\[
    D_\alpha W_{\Psi,\Phi}^{(\alpha)}D_\alpha^{-1}
    =
    W_{\psi,\varphi},
\]
where
\[
    \psi(z)=\Psi\left(\frac z{\sqrt\alpha}\right),
    \qquad
    \varphi(z)=\sqrt\alpha\,\Phi\left(\frac z{\sqrt\alpha}\right).
\]
Consequently, it suffices to study weighted composition operators on $\F$;
the corresponding results for weighted Fock spaces then follow by unitary
equivalence.

We reduce the Main Question to a simpler form through three steps.
The boundedness theorem, together with a zero-free argument, determines the
possible symbols in the nonconstant case.  The elementary cases, which are
already known to be complex symmetric, are then excluded from further
consideration.  Weyl unitary operators and rotations are used to move the
fixed point of the affine symbol to the origin and to normalize the quadratic
coefficient of the weight.  At each step of the reduction, complex symmetry
is preserved under unitary equivalence and multiplication by a nonzero scalar.
\subsection{Step 1: The form of the symbols}

We first recall the boundedness criteria for weighted
composition operators on the Fock space. The following result combines
\cite[Proposition~2.1 and Theorem~2.2]{Le2014} and is stated
here in our notation for later use.
\begin{proposition}\label{prop:boundedness}
Let $\psi$ and $\varphi$ be entire functions on
$\C$ such that $\psi$ is not identically zero.
\(W_{\psi,\varphi}\) is bounded on
\(\F\) if and only if $\psi\in \F,$
and
\begin{align}
 \sup_{z\in\C}|\psi(z)|^2
  \exp\left(|\varphi(z)|^2-|z|^2\right)<\infty .
\end{align}
In this case, $\varphi(z)=az+d$ with $|a|\leq1.$
If $|a|=1,$ then
\(\psi(z)=\eta e^{-a\overline d\,z}\) for some
\(\eta\in\mathbb C\setminus\{0\}\).
\end{proposition}

Whenever \(W_{\psi,\varphi}\) is bounded, the adjoint acts on kernel functions by
\begin{equation}\label{eq:adjoint-kernel}
  W_{\psi,\varphi}^*K_w
  =\overline{\psi(w)}K_{\varphi(w)},\qquad w\in\C.
\end{equation}

A useful property of complex symmetric weighted composition operators is that
their weight functions are zero-free; see \cite[Theorem 3.1]{HaiKhoi2016}.
\begin{lemma}\label{lem:weights}
Suppose that \(\varphi\) is nonconstant and $\psi$ is not identically zero.
If \(W_{\psi,\varphi}\) is complex symmetric, then \(\psi\) is zero-free on  \(\mathbb C\).
\end{lemma}
Let $\mathbb{H}$ be a Hilbert space. For $\bm{x},\bm{y}\in\mathbb{H}$, define the
bounded linear operator $\bm{x} \otimes \bm{y}$ on $\mathbb{H}$ by
\[
    (\bm{x} \otimes \bm{y})f=\langle f,\bm{y}\rangle \bm{x},
    \qquad f\in\mathbb{H}.
\]
If $\bm{x}\ne0$ and $\bm{y}\ne0$, then $\bm{x} \otimes \bm{y}$ has rank one.
\begin{lemma}\label{rank1}
Let \(d\in\mathbb C\), set \(\Phi\equiv d\) and $\Psi\not\equiv0$.
Then
\(W_{\Psi,\Phi}^{(\alpha)}=\Psi\otimes K_d^{(\alpha)}\) is a rank-one
operator and is complex symmetric.
\end{lemma}
\begin{proof}
It follows from \cite[Corollary 5]{GarciaWogen2010} that every rank-one
operator is complex symmetric.  
If \(\Phi\equiv d\), then
\[
\begin{aligned}
\bigl(W_{\Psi,\Phi}^{(\alpha)}f\bigr)(z)
&=\Psi(z)f\bigl(\Phi(z)\bigr)\\
&=\Psi(z)f(d)\\
&=\bigl\langle f,K_d^{(\alpha)}\bigr\rangle_\alpha\Psi(z)\\
&=\bigl((\Psi\otimes K_d^{(\alpha)})f\bigr)(z),
  \quad f\in\mathcal F_\alpha^2.
\end{aligned}
\]
Since $\Psi=W_{\Psi,\Phi}^{(\alpha)}1\in\mathcal F_\alpha^2$ and both $\Psi$ and $K_d^{(\alpha)}$ are nonzero, the identity
$W_{\Psi,\Phi}^{(\alpha)}=\Psi\otimes K_d^{(\alpha)}$
shows that $W_{\Psi,\Phi}^{(\alpha)}$ is a rank-one operator and hence is complex symmetric.
\end{proof}

Izuchi \cite{Izuchi2005} proved that an entire function
\(f\in\mathcal F_{1/2}^{2}\) is cyclic if and only if it is nonvanishing on
\(\mathbb C\). Via the dilation
\(f(z)\mapsto f(\sqrt{2}\,z)\),
this result can be reformulated for the Fock space \(\F\) considered here.
Under this transformation, the condition \(|\alpha|<1/4\) becomes
\(|\beta|<1/2\). Hence we obtain the following proposition; see
\cite[Theorem~1.1 and Corollary~1.2]{Izuchi2005}.

Recall that \(f\in\mathcal F^2\) is cyclic if \
$\overline{\operatorname{span}}\{z^nf:n\ge0\}=\mathcal F^2.$
\begin{proposition}\label{prop:izuchi-cyclic}
Let $f$ be an entire function. Then the following are equivalent:
\begin{itemize}
\item[\textup{(i)}] $f\in\F$ and $f$ is zero-free on $\mathbb C$.
\item[\textup{(ii)}]
$f(z)=\exp(\beta z^{2}+\gamma z+\delta)$
for some $\beta,\gamma,\delta\in\mathbb C$ with
$|\beta|<\frac12$.
\item[\textup{(iii)}]$f$ is cyclic in $\F$.
\end{itemize}
\end{proposition}

\begin{proposition}\label{prop:cs-symbol-form}
Let \(W_{\psi_1,\varphi_1}\) be a nonzero bounded complex symmetric
weighted composition operator on \(\F\). If \(\varphi_1\) is nonconstant,
then there exist constants \(a,d,\beta,\gamma,\delta\in\C\) such that
\begin{equation}\label{eq:1-form}
\begin{gathered}
  \varphi_1(z)=az+d,\qquad 0<|a|\leq 1,\\
  \psi_1(z)=\exp\bigl(\beta z^2+\gamma z+\delta\bigr),
  \qquad |\beta|<\frac12,
\end{gathered}
\qquad z\in\C.
\end{equation}
\end{proposition}

\begin{proof}
By Proposition~\ref{prop:boundedness},
\(\varphi_1(z)=az+d\) for some \(a,d\in\C\) with \(|a|\leq1\).
Since \(\varphi_1\) is nonconstant, \(a\neq0\), so \(0<|a|\leq1\).
Since \(\psi_1=W_{\psi_1,\varphi_1}1\in \F\),
by Lemma~\ref{lem:weights}, the function \(\psi_1\) is zero-free on $\mathbb C$.
The conclusion now follows from Proposition~\ref{prop:izuchi-cyclic}.
\end{proof}

The boundedness and compactness of weighted composition operators with
symbols of the form \eqref{eq:1-form} were characterized in
\cite[Theorem 3.2]{CarrollGilmore2021}.
\begin{proposition}\label{prop:explicit-boundedness}
Let \(W_{\psi_1,\varphi_1}\) be the weighted composition operator
whose symbols are of the form \eqref{eq:1-form} and satisfy
\(0<|a|<1\). Then the following statements hold.
\begin{enumerate}
\item[\textup{(i)}]
The operator \(W_{\psi_1,\varphi_1}\) is compact on \(\F\) if and only if
\begin{equation}\label{eq:compact-condition}
2|\beta|<1-|a|^2.
\end{equation}

\item[\textup{(ii)}]
The operator \(W_{\psi_1,\varphi_1}\) is bounded but not compact on
\(\F\) if and only if
\begin{equation}\label{eq:critical-condition}
2|\beta|=1-|a|^2
\end{equation}
and
\begin{equation*}
|\beta|(\gamma+a\overline d)
+\beta\,\overline{\gamma+a\overline d}=0.
\end{equation*}
\end{enumerate}
\end{proposition}
\subsection{Step 2: Elementary complex symmetric cases}

\begin{lemma}\label{nozero}
If \(\lambda\ne0\), then an operator \(T\) is complex symmetric if and only if
\(\lambda T\) is complex symmetric.
\end{lemma}
In \eqref{eq:1-form},
\[
  W_{\psi_1,\varphi_1}
  =e^\delta W_{\exp(\beta z^2+\gamma z),\varphi_1}.
\]
By Lemma~\ref{nozero}, \(W_{\psi_1,\varphi_1}\) is complex symmetric if and
only if \(W_{\exp(\beta z^2+\gamma z),\varphi_1}\) is complex symmetric.

It follows from \cite[Section 4.1]{GarciaPutinar2006} that every normal operator is complex symmetric.
For reference, we recall the characterization of normal weighted composition operators
\(W_{\psi,\varphi}\) from \cite[Theorem 3.3]{Le2014}.
\begin{proposition}\label{normal}
Let $\psi$ and $\varphi$ be entire functions with $\psi\not\equiv 0.$
Then $W_{\psi,\varphi}$ is a normal bounded operator on $\mathcal F^2$ if and only if
one of the following two cases occurs.

\begin{enumerate}
\item $\varphi(z)=az+d$, where $|a|=1$, and
$\psi(z)=\eta \exp\{-a\overline d\,z\},$ for some $\eta\in\mathbb C\setminus\{0\};$
\item $\varphi(z)=az+d$, where $|a|<1$, and
$\psi(z)
=\eta
\exp\left\{
\frac{1-a}{1-\overline a}\,\overline d\,z
\right\},$ for some $\eta\in\mathbb C\setminus\{0\}.$
\end{enumerate}
\end{proposition}
By Proposition~\ref{prop:boundedness}, the case \(|a|=1\) is contained in
the normal case described in Proposition~\ref{normal}. Having settled the
rank-one case \(a=0\) and the unimodular case \(|a|=1\), it remains to consider
weighted composition operators
whose symbols have the form
\begin{equation}\label{eq:2-form}
\begin{gathered}
  \varphi_2(z)=az+d,\qquad 0<|a|<1,\\
  \psi_2(z)=\exp(\beta z^2+\gamma z),\qquad |\beta|<\frac12.
\end{gathered}
\end{equation}

\subsection{Step 3: Unitary reduction}

In this section, we use the following two unitary operators to transform
\eqref{eq:2-form} into a simpler form.

For \(p\in\C\), the Weyl unitary operator \cite[p.~76]{Zhu2012} is
\begin{equation*}
  (U_pf)(z)=k_p(z)f(z-p),\qquad f\in \F.
\end{equation*}
A direct calculation gives \(U_p^*=U_{-p}\).  For \(|\eta|=1\), let
\[
  R_\eta f(z)=f(\eta z), \qquad f\in \F .
\]
Then \(R_\eta\) is unitary and \(R_\eta^*=R_{\overline\eta}\).

\begin{lemma}\label{lem:reduction}
Let \(\psi_2\) and \(\varphi_2\) be as in \eqref{eq:2-form}. Set
\(p=d/(1-a)\). If $\beta=0$, take $\eta=1$. If $\beta\neq0$, choose a unimodular
constant $\eta$ such that
\begin{align}\label{u}
\eta^{2}\frac{\beta}{1-a^{2}}
=
\left|\frac{\beta}{1-a^{2}}\right|.
\end{align}
Then
\begin{enumerate}
  \item[\textup{(i)}] \[
R_{\eta}U_{-p}W_{\psi_2,\varphi_2}U_pR_{\overline{\eta}}
=
\psi_2(p)W_{\psi,\varphi},
\]
where
\begin{equation}\label{eq:translated-weight}
\begin{gathered}
\varphi(z)=az,\qquad
\psi(z)=\exp\!\left(uz^2+vz\right),\\
u=\beta\eta^2,\qquad
v=\eta\left(2\beta p+\gamma+(a-1)\overline p\right).
\end{gathered}
\end{equation}
\item [\textup{(ii)}] $W_{\psi_2,\varphi_2}$ is complex symmetric if and only if $W_{\psi,\varphi}$ is complex symmetric.
\end{enumerate}

\end{lemma}

\begin{proof}
Let \(f\in\F\) and \(z\in\mathbb C\). Using
\(U_p^*=U_{-p}\) and \(R_\eta^*=R_{\overline{\eta}}\), we obtain
\begin{align*}
&\bigl(R_\eta U_{-p}W_{\psi_2,\varphi_2}
       U_pR_{\overline{\eta}}f\bigr)(z)\\
&\quad=
k_{-p}(\eta z)\psi_2(\eta z+p)
k_p\!\left(\varphi_2(\eta z+p)\right)
f\!\left(
\overline{\eta}\bigl(\varphi_2(\eta z+p)-p\bigr)
\right).
\end{align*}
Since \(\varphi_2(z)=az+d\) and \(p=d/(1-a)\), we have
\(ap+d=p\), and hence
\[
\varphi_2(\eta z+p)=a\eta z+p,
\qquad
\overline{\eta}\bigl(\varphi_2(\eta z+p)-p\bigr)=az.
\]
Moreover, since
$
k_p(w)=\exp(\overline p\,w-\frac{|p|^2}{2}),
$ and $
\psi_2(w)=\exp\!\left(\beta w^2+\gamma w\right),
$
a direct calculation gives
\begin{align*}
&k_{-p}(\eta z)\psi_2(\eta z+p)k_p(a\eta z+p)\\
&\quad=
\exp\!\left\{
\beta\eta^2z^2
+\eta\left(2\beta p+\gamma+(a-1)\overline p\right)z
+\beta p^2+\gamma p
\right\}\\
&\quad=
\psi_2(p)\exp\!\left(uz^2+vz\right).
\end{align*}
Therefore,
\[
\bigl(R_\eta U_{-p}W_{\psi_2,\varphi_2}
U_pR_{\overline{\eta}}f\bigr)(z)
=
\psi_2(p)\psi(z)f(\varphi(z)).
\]

Since \(\psi_2(p)\ne0\), Lemma~\ref{nozero} shows that multiplication by
the scalar \(\psi_2(p)\) does not affect complex symmetry. This proves
\textup{(ii)}.

\end{proof}

\begin{lemma}\label{lem:reduced-c-range}
Let \(W_{\psi,\varphi}\) be the bounded complex symmetric weighted
composition operator obtained in Lemma~\ref{lem:reduction}, where
\[
\varphi(z)=az,\qquad
\psi(z)=\exp(uz^2+vz),\qquad
0<|a|<1.
\]
Set
\[
b=\frac{v}{1-a},
\qquad
c=\frac{u}{1-a^2}.
\]
Then
\[
q(z)=\exp(cz^2+bz)\in\F,
\qquad
0\le c<\frac12.
\]
\end{lemma}

\begin{proof}
Since \(\psi(0)=1\) and \(\varphi(0)=0\),
\eqref{eq:adjoint-kernel} yields
$W_{\psi,\varphi}^{*}K_0=K_0.$
Let \(C\) be a conjugation such that \(W_{\psi,\varphi}\) is
\(C\)-symmetric. Then
$
W_{\psi,\varphi}(CK_0)
=
CW_{\psi,\varphi}^{*}K_0
=
CK_0,
$
and hence \(CK_0\in\ker(W_{\psi,\varphi}-I)\).
Set \(f=CK_0\). Then
\[
\psi(z)f(az)=f(z).
\]
On the other hand, the definitions of \(b\) and \(c\) give
\[
\psi(z)q(az)
=
\exp\!\left((u+a^2c)z^2+(v+ab)z\right)
=
q(z).
\]
Since \(q\) is zero-free, \(h=f/q\) is entire and satisfies
\(h(az)=h(z)\). Thus, for every \(z\in\mathbb C\),
\[
h(z)=h(a^nz)\longrightarrow h(0)
\qquad (n\to\infty),
\]
so \(h\) is constant. Consequently, \(CK_0=\lambda q\) for some
\(\lambda\ne0\), and therefore \(q\in\F\).

By \eqref{eq:translated-weight},
and \eqref{u},
$
c
=
\frac{u}{1-a^2}
=
\eta^2\frac{\beta}{1-a^2}
=
\left|\frac{\beta}{1-a^2}\right|
\ge0.$
Proposition~\ref{prop:izuchi-cyclic} implies \(|c|<1/2\), and hence
\(0\le c<1/2\).
\end{proof}

\subsection{Standing notation}\label{subsec:standing-notation}
Throughout Sections~3 and~4, let \(W_{\psi,\varphi}\) be a
bounded weighted composition operator on \(\mathcal F^2\).
We use the following standing notation and conventions:
\begin{equation}
\begin{gathered}
    \varphi(z)=az,\qquad 0<|a|<1,\\
    \psi(z)=\exp(uz^2+vz),\qquad |u|<\frac12,\\
    b=\frac{v}{1-a}=b_1+ib_2,\qquad
    c=\frac{u}{1-a^2}\in\left[0,\frac12\right),
\end{gathered}
\end{equation}
where
\(b_1=\operatorname{Re}b\) and \(b_2=\operatorname{Im}b\).
Define
\begin{equation*}
  q(z)=\exp(cz^2+bz)\in\F,
  \qquad
  q_n(z)=z^nq(z),
  \qquad n\geq0,
\end{equation*}
and introduce the generating function
\begin{equation*}
  Q(t,z)
  =\e^{tz}q(z)
  =\sum_{n=0}^{\infty}\frac{q_n(z)}{n!}t^n.
\end{equation*}

The polynomials \(h_n\) are defined by
\begin{equation}\label{eq:H-def}
  H(t,z)
  =\exp\{t(z-\overline b)-ct^2\}
  =\sum_{n=0}^{\infty}\frac{h_n(z)}{n!}t^n,
\end{equation}
or equivalently by
\begin{equation}\label{eq:hn-formula}
  h_n(z)
  =\sum_{j=0}^{\lfloor n/2\rfloor}
  \frac{n!}{j!(n-2j)!}(-c)^j
  (z-\overline b)^{n-2j}.
\end{equation}

We further write
\begin{equation}\label{ABD}
\begin{gathered}
  A=1-2c,\qquad B=1+2c,\qquad D=AB=1-4c^2,\\
  M=\exp\left(\frac{b_1^2}{A}+\frac{b_2^2}{B}\right),
\end{gathered}
\end{equation}
and
\begin{equation}\label{eq:m-algebraic}
  L=\frac{b_1}{A}-i\frac{b_2}{B}
   =\frac{\overline b+2cb}{D}.
\end{equation}
\section{Eigenvalues, eigenspaces, and Gram matrices}

Throughout this section, the standing notation and conventions introduced in
Subsection~\ref{subsec:standing-notation} are in force.
We construct two
families of eigenvectors, \(\{q_n\}_{n\geq0}\) for \(W_{\psi,\varphi}\) and
\(\{h_n\}_{n\geq0}\) for \(W_{\psi,\varphi}^*\), prove their completeness,
and compute the Gram matrices needed in the proof of
the main theorem.

When \(c>0\), the polynomials \(h_n\) in \eqref{eq:hn-formula} may also be
identified with translated and rescaled Hermite polynomials $\mathcal H_n$. More precisely,
\[
  h_n(z)
  =
  c^{n/2}
  \mathcal H_n\left(
    \frac{z-\overline b}{2\sqrt c}
  \right),
  \qquad n\geq0,
\]
where
\[
  \mathcal H_n(x)
  =
  n!
  \sum_{j=0}^{\lfloor n/2\rfloor}
  \frac{(-1)^j(2x)^{n-2j}}
       {j!(n-2j)!},
\]
equivalently characterized by
\[
  \exp(2x\tau-\tau^2)
  =
  \sum_{n=0}^{\infty}
  \mathcal H_n(x)\frac{\tau^n}{n!};
\]
see \cite[Chapter 11]{Rainville1960}.
When \(c=0\), one simply has
\[
  h_n(z)=(z-\overline b)^n.
\]
\begin{lemma}\label{lem:norm-convergence-generating}
For every \(t\in\C\), the expansions
\[
  Q(t,z)
  =\e^{tz}q(z)
  =\sum_{n=0}^{\infty}\frac{q_n(z)}{n!}t^n
\]
and
\[
  H(t,z)
  =\exp\{t(z-\overline b)-ct^2\}
  =\sum_{n=0}^{\infty}\frac{h_n(z)}{n!}t^n
\]
converge in the norm of \(\F\). In fact, both series converge locally
uniformly in \(t\) with respect to the \(\F\)-norm.
\end{lemma}

\begin{proof}
Fix \(R>0\), and set
\[
  Q_N(t,z)=\sum_{n=0}^{N}\frac{q_n(z)}{n!}t^n,
  \qquad
  H_N(t,z)=\sum_{n=0}^{N}\frac{h_n(z)}{n!}t^n.
\]

Since \(q_n(z)=z^nq(z)\), for \(|t|\leq R\),
\[
  |Q(t,z)-Q_N(t,z)|
  \leq |q(z)|\rho_N(z),
\]
where
\[
  \rho_N(z)
  =
  \sum_{n=N+1}^{\infty}\frac{(R|z|)^n}{n!}.
\]
For every fixed \(z\in\C\),
$\rho_N(z)\longrightarrow0\ \text{as }N\to\infty,$
while \(\rho_N(z)\leq\e^{R|z|}\). Hence
\[
\begin{aligned}
  \sup_{|t|\leq R}
  |Q(t,z)-Q_N(t,z)|^2\e^{-|z|^2}
  &\leq
  |q(z)|^2\e^{2R|z|}\e^{-|z|^2}\\
  &\leq
  \exp\left\{
    (2c-1)|z|^2+2(|b|+R)|z|
  \right\}.
\end{aligned}
\]
The last function is integrable over \(\C\), since \(c<1/2\).
Therefore, by the dominated convergence theorem,
\[
\begin{aligned}
  \sup_{|t|\leq R}
  \|Q(t,\cdot)-Q_N(t,\cdot)\|_{\F}^{2}
  \leq
  \frac{1}{\pi}\int_{\C}
  |q(z)|^2\rho_N(z)^2\e^{-|z|^2}\,\dd A(z)
  \longrightarrow0,\quad \text{as }N\to\infty.
\end{aligned}
\]

Formula~\eqref{eq:hn-formula} gives
\[
\begin{aligned}
  \sum_{n=0}^{\infty}\frac{|h_n(z)|}{n!}R^n
  &\leq
  \sum_{j=0}^{\infty}\sum_{k=0}^{\infty}
  \frac{(cR^2)^j}{j!}
  \frac{(R|z-\overline b|)^k}{k!}
=
  \exp\left\{cR^2+R|z-\overline b|\right\}.
\end{aligned}
\]
Define
\[
  \sigma_N(z)
  =
  \sum_{n=N+1}^{\infty}\frac{|h_n(z)|}{n!}R^n.
\]
Then, for every fixed \(z\in\C\),
$\sigma_N(z)\longrightarrow0
\quad \text{as }N\to\infty,$
and, for \(|t|\leq R\),
\[
  |H(t,z)-H_N(t,z)|\leq\sigma_N(z).
\]
Moreover,
\[
\begin{aligned}
  \sigma_N(z)^2\e^{-|z|^2}
  &\leq
  \exp\left\{
    2cR^2+2R|z-\overline b|-|z|^2
  \right\}\\
  &\leq
  \exp\left\{
    2cR^2+2R|b|+2R|z|-|z|^2
  \right\},
\end{aligned}
\]
and the last function is integrable over \(\C\). Thus another
application of the dominated convergence theorem yields
\[
\begin{aligned}
  \sup_{|t|\leq R}
  \|H(t,\cdot)-H_N(t,\cdot)\|_{\F}^{2}
  &\leq
  \frac{1}{\pi}\int_{\C}
  \sigma_N(z)^2\e^{-|z|^2}\,\dd A(z)
  \longrightarrow0,\qquad\text{as }N\to\infty.
\end{aligned}
\]

Since \(R>0\) was arbitrary, both expansions converge, as
\(N\to\infty\), locally uniformly in \(t\) with respect to the
\(\F\)-norm.
\end{proof}
\begin{lemma}\label{dense}
Both \(\operatorname{span}\{h_n:n\geq0\}\) and
\(\operatorname{span}\{q_n:n\geq0\}\) are dense in \(\F\).
\end{lemma}
\begin{proof}
Each \(h_n\) is a monic polynomial of degree \(n\); hence the \(h_n\)'s span
all polynomials and are complete in \(\F\).  The completeness of the
\(\{q_n\}\) follows from Proposition~\ref{prop:izuchi-cyclic}, since \(q\) is a cyclic vector in \(\F\).
\end{proof}
\begin{theorem}\label{lem:eigenvectors}
For every \(n\geq0\),
\begin{equation}\label{eq:eigenspace-W}
  \ker(W_{\psi,\varphi}-a^nI)=\operatorname{span}\{q_n\},
\end{equation}
and
\begin{equation}\label{eq:adjoint-eigenvectors}
\ker(W_{\psi,\varphi}^{*}-\bar a^{\,n}I)
=\operatorname{span}\{h_n\}.
\end{equation}
\end{theorem}

\begin{proof}
The identities \(u+c a^2=c\) and \(v+ba=b\) give
\[
 \bigl(W_{\psi,\varphi}q_n\bigr)(z)
 =\e^{uz^2+vz}(az)^n\e^{ca^2z^2+baz}
 =a^nq_n(z),\qquad n=0,1,2,\ldots.
\]
If \(W_{\psi,\varphi}g=a^ng\), then \(\omega =g/q\) is entire and
\[\omega (az)=\frac{\psi(z)g(az)}{\psi(z)q(az)}=\frac{ W_{\psi,\varphi}g(z)}{W_{\psi,\varphi}q(z)}=\frac{a^n g(z)}{q(z)}=a^n\omega (z).\]
Comparing the Taylor coefficients of \(\omega \), and using
\(0<|a|<1\), shows that \(\omega \) is a scalar multiple of \(z^n\).  This proves
\eqref{eq:eigenspace-W}.

 To prove
\eqref{eq:adjoint-eigenvectors}, first write
\[
  H(t,z)=\e^{-t\overline b-c t^2}K_{\bar{t}}(z).
\]
Equation \eqref{eq:adjoint-kernel} and the identities
\(\bar v=(1-\bar a)\bar b\) and
\(\bar u=(1-\bar a^2)c\) yield
\[
\begin{aligned}
 W_{\psi,\varphi}^*H(t,z)
 &=\e^{-t\overline b-c t^2}
   \e^{\overline u t^2+\overline v t}K_{a\overline t}(z)\\
 &=\exp\{\overline a t(z-\overline b)
          -c \bar a^2t^2\}\\
 &=H(\bar a t,z).
\end{aligned}
\]
On the other hand,
\[
    H(\bar a t,z)
    =
    \sum_{n=0}^{\infty}
    \frac{h_n(z)}{n!}(\bar a t)^n
    =
    \sum_{n=0}^{\infty}
    \frac{\bar a^{\,n}h_n(z)}{n!}t^n.
\]
By Lemma~\ref{lem:norm-convergence-generating}, we have
\[
    W_{\psi,\varphi}^{*}H(t,z)
    =
    \sum_{n=0}^{\infty}
    \frac{W_{\psi,\varphi}^{*}h_n(z)}{n!}t^n.
\]
Hence
$W_{\psi,\varphi}^{*}h_n=\bar a^{\,n}h_n,$
and 
$
\operatorname{span}\{h_n\}
\subset \ker\left(W_{\psi,\varphi}^{*}-\overline a^{\,n}I\right).
$
It remains to prove the reverse inclusion.

Let
$f\in \ker(W_{\psi,\varphi}^{*}-\bar a^{\,n}I).$
Then
$W_{\psi,\varphi}^{*}f=\overline a^{\,n}f.$
For every \(m\geq0\), since
$W_{\psi,\varphi}q_m=a^m q_m,$
we have
\[
\begin{aligned}
    a^m\langle q_m,f\rangle
    &=
    \langle W_{\psi,\varphi}q_m,f\rangle  \\
    &=
    \langle q_m,W_{\psi,\varphi}^{*}f\rangle  \\
    &=
    \langle q_m,\overline a^{\,n}f\rangle  \\
    &=
    a^n\langle q_m,f\rangle .
\end{aligned}
\]
Thus
$(a^m-a^n)\langle q_m,f\rangle=0.$
Since \(0<|a|<1\), we have \(a^m\neq a^n\) whenever \(m\neq n\). Hence
$\langle q_m,f\rangle=0,\ m\neq n.$
Therefore $f\in \{q_m:m\neq n\}^{\perp}.$

We claim that \(\{q_m:m\neq n\}^{\perp}\) is at most one-dimensional.
Set
$X_n:=\{q_m:m\neq n\}^{\perp},$
and define the linear functional
\[
  \Lambda_n:X_n\longrightarrow\mathbb C,
  \qquad
  \Lambda_n(g)=\langle g,q_n\rangle.
\]
If \(g\in\ker\Lambda_n\), then \(g\perp q_n\). Since \(g\in X_n\), we
also have \(g\perp q_m\) for every \(m\neq n\). Hence \(g\perp q_m\) for
all \(m\geq0\). The completeness of \(\{q_m:m\geq0\}\) implies \(g=0\),
so \(\Lambda_n\) is injective. Since \(\Lambda_n\) maps \(X_n\)
injectively into the one-dimensional space \(\mathbb C\), it follows that
$\dim X_n\leq1.$

Since \(h_n\neq0\) and
$
h_n\in \ker\left(W_{\psi,\varphi}^{*}-\overline a^{\,n}I\right),
$
the eigenspace is nonzero and has dimension at most one. Therefore
\[
    \ker\left(W_{\psi,\varphi}^{*}-\bar a^{\,n}I\right)
    =
    \operatorname{span}\{h_n\}.
\]
\end{proof}

\begin{corollary}\label{cor:eigen}
\[\sigma_p(W_{\psi, \varphi})=\{a^n:n\ge0\},\qquad \sigma_p(W_{\psi, \varphi}^*)
    =
    \{\overline a^{\,n}:n\ge0\}.\]
\end{corollary}
\begin{proof}
By \eqref{eq:eigenspace-W}, we have
$\{a^n\}_{n=0}^\infty\subset \sigma_p(W_{\psi, \varphi}).$

Let $\lambda\in \sigma_p(W_{\psi, \varphi})$, and choose a nonzero
\(f\in\F\) such that $W_{\psi, \varphi}f=\lambda f$.
Then \(\psi(z)f(az)=\lambda f(z)\).
Write \(f(z)=z^j\theta(z)\), where $j\geq0$ and \(\theta(0)\ne0\). Then
$\psi(z)a^jz^j\theta(az)=\lambda z^j \theta(z).$
That is
$a^j\psi(z)\theta(az)=\lambda \theta(z).$
Setting $z=0$ gives
$\lambda=a^j\psi(0)=a^j.$

By  \eqref{eq:adjoint-eigenvectors}, we have
$\{\bar{a}^n\}_{n=0}^\infty\subset \sigma_p(W_{\psi, \varphi}^*).$
Suppose $W_{\psi, \varphi}^*g=\lambda g$
for some nonzero \(g\in\mathcal F^2\). For every \(n\geq0\), we have
\[
\begin{aligned}
    a^n\langle q_n,g\rangle
    &=
    \langle W_{\psi, \varphi}q_n,g\rangle  \\
    &=
    \langle q_n,W_{\psi, \varphi}^*g\rangle  \\
    &=
    \langle q_n,\lambda g\rangle  \\
    &=
    \overline\lambda\langle q_n,g\rangle.
\end{aligned}
\]
Therefore
\[(a^n-\overline\lambda)\langle q_n,g\rangle=0,
    \qquad n\geq 0.\]
If \(\lambda\neq \overline a^{\,n}\) for every \(n\), then
$\langle q_n,g\rangle=0,\ n\geq 0.$
Since \(\{q_n\}\) is complete in \(\mathcal F^2\), this implies \(g=0\),
a contradiction.  Hence
$\lambda=\bar a^{\,n}$
for some \(n\geq0\).  Thus
\[
    \sigma_p(W_{\psi, \varphi}^*)
    =
    \{\bar a^{n}:n\ge0\}.
\]
\end{proof}

\begin{corollary}\label{spectrum}
\[
\sigma(W_{\psi,\varphi})
=
\{0\}\cup\{a^n:n\ge0\}.
\]
\end{corollary}
\begin{proof}
By Corollary~\ref{cor:eigen},
\[
\sigma_p(W_{\psi,\varphi})
=
\{a^n:n\ge0\}.
\]
Since \(a^n\to0\) and the spectrum is closed,
$0\in\sigma(W_{\psi,\varphi}).$

If \(W_{\psi,\varphi}\) is compact, then the Riesz--Schauder theorem gives
\[
\sigma(W_{\psi,\varphi})
=
\{0\}\cup\sigma_p(W_{\psi,\varphi})
=
\{0\}\cup\{a^n:n\ge0\}.
\]

Assume now that \(W_{\psi,\varphi}\) is not compact. We claim that
\(a\notin\mathbb R\). Indeed, if \(a\in\mathbb R\), then
\[2|u|=2c(1-a^2)<1-a^2=1-|a|^2,\]
so Proposition~\ref{prop:explicit-boundedness} \textup{(i)} would imply
that \(W_{\psi,\varphi}\) is compact, a contradiction. Hence
\(a\notin\mathbb R\). Since \(W_{\psi,\varphi}\) is bounded but not
compact, Proposition~\ref{prop:explicit-boundedness} \textup{(ii)} yields
$2|u|=1-|a|^2.$
Moreover,
\[
W_{\psi,\varphi}^2=W_{f,a^2z},
\qquad
f(z)=\exp\!\left(u(1+a^2)z^2+v(1+a)z\right).
\]
Since \(a\notin\mathbb R\),
$|1+a^2|<1+|a|^2,$
and therefore
\[
2|u(1+a^2)|
=(1-|a|^2)|1+a^2|
<1-|a|^4.
\]
Applying Proposition~\ref{prop:explicit-boundedness} to
\(W_{f,a^2z}\), we obtain that \(W_{\psi,\varphi}^2\) is compact.
Thus \(W_{\psi,\varphi}\) is power compact,  so every nonzero spectral value is an eigenvalue. Consequently,
\[
\sigma(W_{\psi,\varphi})
=
\{0\}\cup\sigma_p(W_{\psi,\varphi})
=
\{0\}\cup\{a^n:n\ge0\}.
\]
\end{proof}
\begin{remark}
The compactness of $W_{\psi,\varphi}^2$ in the proof of the preceding corollary
also follows from Theorem~4.2 of \cite{CarrollGilmore2021}.
\end{remark}
We will also use the following elementary Gaussian integrals. For
\(\kappa>0\),
\[
  \int_{-\infty}^{+\infty}e^{-\kappa x^2}\,dx
  =\sqrt{\frac{\pi}{\kappa}},
  \quad
  \int_{-\infty}^{+\infty}xe^{-\kappa x^2}\,dx=0,
  \quad
  \int_{-\infty}^{+\infty}x^2e^{-\kappa x^2}\,dx
  =\frac{1}{2\kappa}\sqrt{\frac{\pi}{\kappa}}.
\]
These standard formulas yield the shifted and linearly perturbed Gaussian
identities stated in the following lemma. The shifted identities below use
only real translations. The formula with a complex linear term follows first
for real \(\delta\) by completing the square and then for all
\(\delta\in\C\) by analytic continuation.
\begin{lemma}\label{le:Gaussian-integ}
For \(\kappa>0\), \(\beta\in\mathbb R\), and \(\delta\in\C\), we have
\begin{equation}\label{eq:shifted-gaussian}
   \int_{-\infty}^{+\infty}\e^{-\kappa (x-\beta)^2}\dd x=\sqrt{\frac{\pi}{\kappa}},
\end{equation}
\begin{equation}\label{eq:linear-gaussian}
   \int_{-\infty}^{+\infty}\e^{-\kappa x^2+\delta x}\dd x
   =\exp\left(\frac{\delta^2}{4\kappa}\right)\sqrt{\frac{\pi}{\kappa}},
\end{equation}
\begin{equation}\label{eq:shifted-first-moment}
  \int_{-\infty}^{+\infty}x\e^{-\kappa (x-\beta)^2}\dd x
  =\beta\int_{-\infty}^{+\infty}\e^{-\kappa (x-\beta)^2}\dd x,
\end{equation}
\begin{equation}\label{eq:shifted-second-moment}
   \int_{-\infty}^{+\infty}x^2\e^{-\kappa (x-\beta)^2}\dd x
   =\left(\frac{1}{2\kappa}+\beta^2\right)\int_{-\infty}^{+\infty}\e^{-\kappa (x-\beta)^2}\dd x.
\end{equation}
\end{lemma}

The next calculation is the numerical core of the necessity argument.

\begin{lemma}\label{lem:moments}
\begin{equation}\label{eq:moments}
  \|q\|^2=\frac{M}{\sqrt D},\qquad
  \ip{zq}{q}=\|q\|^2L,
  \qquad
  \|zq\|^2=\|q\|^2\left(|L|^2+\frac1D\right).
\end{equation}
The constants $M$, $D$, and $L$ are defined in Subsection~\ref{subsec:standing-notation}.
\end{lemma}

\begin{proof}
For \(z=x+iy\),
\[\begin{aligned}
 |q(z)|^2\e^{-|z|^2}&=
   \e^{-Ax^2-By^2+2b_1x-2b_2y}  \\
 &=M\exp\left[-A\left(x-\frac{b_1}{A}\right)^2
              -B\left(y+\frac{b_2}{B}\right)^2\right].
\end{aligned}
\]
Set
\[
  E(x,y)=\exp\left[-A\left(x-\frac{b_1}{A}\right)^2
              -B\left(y+\frac{b_2}{B}\right)^2\right].
\]
Using the formula \eqref{eq:shifted-gaussian}, first with
\((\kappa,\beta)=(A,b_1/A)\) and then with
\((\kappa,\beta)=(B,-b_2/B)\), we get
\[
\begin{aligned}
  I_0
  &:=\int_{\R}\int_{\R}E(x,y)\dd x\dd y\\
  &=\left(\int_{\R}\e^{-A(x-\frac{b_1}{A})^2}\dd x\right)
    \left(\int_{\R}\e^{-B(y+\frac{b_2}{B})^2}\dd y\right)\\
  &=\sqrt{\frac{\pi}{A}}\sqrt{\frac{\pi}{B}}
    =\frac{\pi}{\sqrt D}.
\end{aligned}
\]
Therefore
\[
\begin{aligned}
   \|q\|^2
   &=
   \frac1\pi
   \int_{\R}\int_{\R}
   |q(x+iy)|^2\e^{-x^2-y^2}\dd x\dd y=\frac{M}{\pi}I_0= \frac{M}{\sqrt D}.
\end{aligned}
\]

Next we compute $\ip{zq}{q}$.  By the  formula
\eqref{eq:shifted-first-moment},
\[
\begin{aligned}
I_x
  :=\int_{\R}\int_{\R}xE(x,y)\dd x\dd y
    =\frac{b_1}{A}I_0,\quad
I_y
:=\int_{\R}\int_{\R}yE(x,y)\dd x\dd y
=-\frac{b_2}{B}I_0.
\end{aligned}
\]
Since \(z=x+iy\), we have
\[\begin{aligned}
\ip{zq}{q}
&= \frac{M}{\pi}\int_{\R}\int_{\R}(x+iy)E(x,y)\dd x\dd y \\
&= \frac{M}{\pi}(I_x+iI_y)\\
&= \frac{M}{\pi}\left(\frac{b_1}{A}-i\frac{b_2}{B}\right)I_0\\
&=\|q\|^2L.
\end{aligned}
\]
By the formula
\eqref{eq:shifted-second-moment},
\[
\begin{aligned}
  I_{xx}
  &:=\int_{\R}\int_{\R}x^2E(x,y)\dd x\dd y
    =\left(\frac{1}{2A}+\frac{b_1^2}{A^2}\right)I_0,\\
  I_{yy}
  &:=\int_{\R}\int_{\R}y^2E(x,y)\dd x\dd y
    =\left(\frac{1}{2B}+\frac{b_2^2}{B^2}\right)I_0.
\end{aligned}
\]
Thus
\[
\begin{aligned}
\|zq\|^2
&= \frac{M}{\pi}
   \int_{\R}\int_{\R}(x^2+y^2)E(x,y)\dd x\dd y\\
&= \frac{M}{\pi}(I_{xx}+I_{yy})\\
&= \|q\|^2
   \left(
     \frac{b_1^2}{A^2}+\frac{b_2^2}{B^2}
     +\frac1{2A}+\frac1{2B}
   \right).
\end{aligned}
\]
Since
\[
  |L|^2=\frac{b_1^2}{A^2}+\frac{b_2^2}{B^2},
  \qquad
  \frac1{2A}+\frac1{2B}
  =\frac{A+B}{2AB}=\frac1D,
\]
we obtain
\[
  \|zq\|^2
  =\|q\|^2\left(|L|^2+\frac1D\right).
\]
\end{proof}

The following is a key lemma of the paper. Since $Q(t,z)
  =\sum_{n=0}^{\infty}\frac{q_n(z)}{n!}t^n$ and $H(t,z)
  =\sum_{n=0}^{\infty}\frac{h_n(z)}{n!}t^n$, it essentially
determines the Gram matrices of the families $\{q_n\}$ and $\{h_n\}$.
\begin{lemma}\label{lem:gram}
For \(s,r\in\C\), the following identities hold:
\begin{align}
 \ip{H(s,\cdot)}{H(\overline r,\cdot)}
 &=\exp\{sr-s\overline b-rb-c(s^2+r^2)\},
 \label{eq:H-Gram}\\
 \ip{Q(s,\cdot)}{Q(\overline r,\cdot)}
 &=\|q\|^2\exp\left\{
   \frac cD(s^2+r^2)+\frac1Dsr+Ls+\overline L r\right\},
 \label{eq:Q-Gram}\\
 \ip{H(s,\cdot)}{Q(\bar r,\cdot)}&=\e^{sr}. \notag
\end{align}
In particular,
\begin{equation}\label{eq:biorthogonal}
  \ip{h_m}{q_n}
=\begin{cases}
n!,
& n=m,\\[1ex]
0,
& n\neq m.
\end{cases}
\end{equation}
\end{lemma}

\begin{proof}
The kernel identity
\[H(s,\cdot)=\e^{-s\overline b-cs^2}K_{\overline s}\qquad\text{and}\qquad
\ip{K_{\overline s}}{K_r}=\e^{sr}\]
gives
\[
\begin{aligned}
 \ip{H(s,\cdot)}{H(\overline r,\cdot)}
 &=\e^{-s\overline b-cs^2}
   \e^{-rb-cr^2}\ip{K_{\overline s}}{K_r}  \\
 &=\exp\{sr-s\overline b-rb-c(s^2+r^2)\}.
\end{aligned}
\]
This proves \eqref{eq:H-Gram}.

We next prove \eqref{eq:Q-Gram}.  Since
\(Q(t,z)=\e^{tz}q(z)\),
\[
 \ip{Q(s,\cdot)}{Q(\overline r,\cdot)}
 =\frac1\pi\int_{\C}\e^{sz+r\overline z}|q(z)|^2\e^{-|z|^2}\dd A(z).
\]
Because for \(z=x+iy\),
\[\begin{aligned}
&\quad \e^{sz+r\overline z}|q(z)|^2\e^{-|z|^2}\\
&=\e^{(s+r)x+i(s-r)y}\e^{-Ax^2-By^2+2b_1x-2b_2y}\\
&=\e^{-Ax^2+(2b_1+(s+r))x-By^2+(-2b_2+i(s-r))y},
\end{aligned}
\]
put
\[
  \delta_x=2b_1+s+r,\qquad
  \delta_y=-2b_2+i(s-r).
\]
Using the linear Gaussian integral \eqref{eq:linear-gaussian} from
Lemma~\ref{le:Gaussian-integ}, first with
\((\kappa,\delta)=(A,\delta_x)\) and then with
\((\kappa,\delta)=(B,\delta_y)\), we obtain
\[\begin{aligned}
&\quad \ip{Q(s,\cdot)}{Q(\overline r,\cdot)}\\
&=\frac{1}{\pi}
  \left(\int_{\R}\e^{-Ax^2+\delta_xx}\dd x\right)
  \left(\int_{\R}\e^{-By^2+\delta_yy}\dd y\right)\\
&=\frac1{\sqrt{AB}}
  \exp\left(\frac{\delta_x^2}{4A}+\frac{\delta_y^2}{4B}\right).
\end{aligned}
\]
Now
\[
\begin{aligned}
\frac{\delta_x^2}{4A}
&=\frac{b_1^2}{A}+\frac{b_1}{A}(s+r)+\frac{(s+r)^2}{4A},\\
\frac{\delta_y^2}{4B}
&=\frac{b_2^2}{B}-i\frac{b_2}{B}(s-r)-\frac{(s-r)^2}{4B}.
\end{aligned}
\]
Therefore, by the definitions of \(M\) and \(D=AB\),
\[
\begin{aligned}
 \ip{Q(s,\cdot)}{Q(\overline r,\cdot)}
&=\frac{M}{\sqrt D}
  \exp\left\{
   (s+r)\frac{b_1}{A}-i(s-r)\frac{b_2}{B}
   +\frac{(s+r)^2}{4A}-\frac{(s-r)^2}{4B}\right\}.
\end{aligned}
\]
Note that
\[
  (s+r)\frac{b_1}{A}-i(s-r)\frac{b_2}{B}
  =Ls+\overline L r,
\]
and 
\[
\begin{aligned}
  \frac{(s+r)^2}{4A}-\frac{(s-r)^2}{4B}
  &=\frac{B(s+r)^2-A(s-r)^2}{4D}  \\
  &=\frac{4c(s^2+r^2)+4sr}{4D}
    =\frac cD(s^2+r^2)+\frac1Dsr.
\end{aligned}
\]
Since \(\|q\|^2=M/\sqrt D\) by \eqref{eq:moments}, this gives
\[
 \ip{Q(s,\cdot)}{Q(\overline r,\cdot)}
 =\|q\|^2
 \exp\left\{\frac cD(s^2+r^2)+\frac1Dsr+Ls+\overline L r\right\}.
\]
This proves
\eqref{eq:Q-Gram}.

Finally, the reproducing property gives
\[
\begin{aligned}
 \ip{H(s,\cdot)}{Q(\overline r,\cdot)}
 =\e^{-s\overline b-cs^2}
   \overline{Q(\overline r,\overline s)}
 =\e^{-s\overline b-cs^2}\e^{sr}
   \e^{cs^2+\overline b s}=\e^{sr}.
\end{aligned}
\]
Since $
Q(\bar{r},z)
=\sum_{n=0}^{\infty}\frac{q_n(z)}{n!}\bar{r}^n$
and
$H(s,z)=\sum_{n=0}^{\infty}\frac{h_n(z)}{n!}s^n,$
applying the formula above and using uniqueness of the resulting two-variable power series proves \eqref{eq:biorthogonal}.
\end{proof}

\section{Proof of Theorem B}

Throughout this section, the standing hypotheses and notation of
Subsection~\ref{subsec:standing-notation} are in force.
\begin{lemma}\label{lm:main}
Let \(W_{\psi,\varphi}\) be a bounded weighted composition operator of the
form \eqref{eq:in-reduced-form}. 
If \(W_{\psi,\varphi}\) is
complex symmetric, then either
\begin{equation}\label{eq:main-condition}
\begin{gathered}
c=0;\\
\text{or}\quad
0<c<\frac12
\quad\text{and}\quad
\frac{(\operatorname{Re}b)^2}{1-2c}
=
\frac{(\operatorname{Im}b)^2}{1+2c}.
\end{gathered}
\end{equation}
\end{lemma}

\begin{proof}
Suppose that \(CW_{\psi,\varphi}^*=W_{\psi,\varphi}C\) for a conjugation
\(C\).  By Lemma~\ref{lem:eigenvectors}, 
\[W_{\psi,\varphi}^*h_0=h_0,\quad\text{and}\quad W_{\psi,\varphi}^*h_1=\bar{a}h_1.\]
Hence \[W_{\psi,\varphi}Ch_0=Ch_0,\quad\text{and}\quad W_{\psi,\varphi}Ch_1=aCh_1.\]
The one-dimensional eigenspaces in \eqref{eq:eigenspace-W} imply that there
exist nonzero constants \(\lambda\) and \(\mu\) such that
\begin{equation*}
  Ch_0=\lambda q,
  \qquad
  Ch_1=\mu zq.
\end{equation*}
Since \(C\) is isometric,
\begin{equation}\label{eq:norm-relations}
  |\lambda|^2\|q\|^2=1,
  \qquad
  |\mu|^2\|zq\|^2=\|z-\overline b\|^2=1+|b|^2.
\end{equation}
Furthermore,
\[
  \mu\overline\lambda\ip{zq}{q}
  =\ip{Ch_1}{Ch_0}=\ip{h_0}{h_1}=-b.
\]
Taking absolute values and using \eqref{eq:norm-relations}, we obtain
\begin{equation*}
  (1+|b|^2)|\ip{zq}{q}|^2
  =|b|^2\|q\|^2\|zq\|^2.
\end{equation*}
Substitution of \eqref{eq:moments} reduces this identity to
$|L|^2=|b|^2/D.$
By \eqref{eq:m-algebraic}, this is
\[
  |\overline b+2cb|^2=(1-4c^2)|b|^2.
\]
Writing \(b=b_1+ib_2\) and expanding both sides gives
\[
  4c\bigl((1+2c)b_1^2-(1-2c)b_2^2\bigr)=0.
\]
Thus either \(c=0\), or \(c>0\) and \eqref{eq:main-condition} holds.
\end{proof}
\begin{lemma}\label{lem:abstract-gram}
Let \(\{x_n\}_{n\ge0}\) and \(\{y_n\}_{n\ge0}\) be complete families in a
Hilbert space \(\mathbb{H}\), and suppose that
\[
\langle y_m,x_n\rangle=\kappa_n\delta_{mn},
\qquad \kappa_n\ne0.
\]
Let \(\lambda_n\ne0\) for all \(n\ge0\). If
\begin{equation}\label{eq:abstract-gram}
\langle \lambda_mx_m,\lambda_nx_n\rangle
=
\langle y_n,y_m\rangle,
\qquad m,n\ge0,
\end{equation}
then the formula
\[
J\!\left(\sum_{n=0}^N\alpha_ny_n\right)
=
\sum_{n=0}^N\overline{\alpha_n}\lambda_nx_n
\]
extends uniquely to a conjugation on \(\mathbb{H}\).

\end{lemma}
\begin{proof}
Let
\[
  \mathcal D=\operatorname{span}\{y_n:n\geq0\}.
\]
Since \(\{y_n\}_{n\geq0}\) is complete, \(\mathcal D\) is dense in
\(\mathbb H\).

Note that the family \(\{y_n\}_{n\geq0}\) is linearly independent.
Indeed, if
$\sum_{n=0}^{N}\eta_n y_n=0,$
then, for each \(0\leq k\leq N\),
\[
  0
  =
  \left\langle
    \sum_{n=0}^{N}\eta_n y_n,x_k
  \right\rangle
  =
  \eta_k\kappa_k.
\]
Since \(\kappa_k\ne0\), it follows that \(\eta_k=0\) for every \(k\).
Hence the formula
\[
J_0\!\left(\sum_{n=0}^N\alpha_n y_n\right)
=
\sum_{n=0}^N\overline{\alpha_n}\lambda_n x_n
\]
defines a conjugate-linear operator \(J_0\) on \(\mathcal D\).

Let \(f=\sum_{m=0}^{M}\alpha_m y_m\in\mathcal D\) and
\(g=\sum_{n=0}^{N}\beta_n y_n\in\mathcal D\).
Using \eqref{eq:abstract-gram}, we obtain
\[
\begin{aligned}
  \langle J_0f,J_0g\rangle
  &=
  \sum_{m=0}^{M}\sum_{n=0}^{N}
  \overline{\alpha_m}\beta_n
  \langle \lambda_mx_m,\lambda_nx_n\rangle
  \\
  &=
  \sum_{m=0}^{M}\sum_{n=0}^{N}
  \overline{\alpha_m}\beta_n
  \langle y_n,y_m\rangle
  \\
  &=
  \left\langle
    \sum_{n=0}^{N}\beta_ny_n,
    \sum_{m=0}^{M}\alpha_my_m
  \right\rangle
  \\
  &=
  \langle g,f\rangle.
\end{aligned}
\]
In particular,
\[
  \|J_0f\|^2
  =
  \langle J_0f,J_0f\rangle
  =
  \langle f,f\rangle
  =
  \|f\|^2.
\]
Hence \(J_0\) is an isometry on \(\mathcal D\).
It therefore extends uniquely to an anti-linear isometry
$J:\mathbb H\longrightarrow\mathbb H.$
By the polarization identity, we have
\begin{equation}\label{eq:J-antiunitary}
  \langle J\bm{x},J\bm{y}\rangle=\langle \bm{y},\bm{x}\rangle
\end{equation}
for all \(\bm{x},\bm{y}\in\mathbb H\).

Fix \(m,n\geq0\). Since \(Jy_m=\lambda_mx_m\), it follows from
\eqref{eq:J-antiunitary} that
\[
\begin{aligned}
  \langle J^2y_n,\lambda_mx_m\rangle
  &=
  \langle J(\lambda_nx_n),Jy_m\rangle
  \\
  &=
  \langle y_m,\lambda_nx_n\rangle
  \\
  &=
  \overline{\lambda_n}\kappa_n\delta_{mn}\\
  &=\langle y_n,\lambda_mx_m\rangle.
\end{aligned}
\]
Completeness of $\{x_m\}$ implies $J^2y_n=y_n$ for every $n$, and completeness of $\{y_n\}$ gives $J^2=I$.
\end{proof}
\begin{lemma}\label{lem:construction}
Assume that \eqref{eq:main-condition} holds, and choose \(\rho\) by
\begin{equation}\label{eq:rho}
  \rho=
  \begin{cases}
    -\dfrac{b}{L}
    \,
      & b\ne0,\\[3mm]
    1,
      & b=0,\ c=0,\\[1mm]
    i\sqrt D,
      & b=0,\ c>0.
  \end{cases}
\end{equation}
Define
\begin{equation}\label{eq:lambda-n}
  \lambda_n=\frac{\rho^n}{\|q\|},
  \qquad n\geq0.
\end{equation}
Then the formula
\begin{equation}\label{eq:J-definition}
  J\left(\sum_{n=0}^N\alpha_nh_n\right)
  =\sum_{n=0}^N\overline{\alpha_n}\lambda_nq_n
\end{equation}
extends uniquely to a conjugation on \(\F\).
\end{lemma}

\begin{proof}
We divide the proof into two steps.

\medskip
\noindent\emph{Step 1.}
We first prove that, for all $r,s\in\C$,
\begin{equation}\label{eq:exponent-match}
\begin{aligned}
\frac{c}{D}\bigl(\rho^2s^2+\overline\rho^{\,2}r^2\bigr)
   +\frac{|\rho|^2}{D}sr
   +L\rho s+\overline L\,\overline\rho r 
=sr-sb-r\overline b-c(s^2+r^2).
\end{aligned}
\end{equation}
If $c=0$, then $D=1$ and $L=\overline b$.  By the choice of
$\rho$ in \eqref{eq:rho},
\[
  |\rho|=1,
  \qquad
 L \rho =-b,
\]
and \eqref{eq:exponent-match} follows immediately.

Now suppose that $c>0$.  If $b=0$, then $L=0$ and
$\rho=i\sqrt D$, so
\[
  \rho L=-b,
  \qquad
  |\rho|^2=D,
  \qquad
  \rho^2=-D.
\]
If $b\ne0$, then \eqref{eq:main-condition} gives
  $(1+2c)b_1^2=(1-2c)b_2^2.$
Since
$\overline b+2cb=(1+2c)b_1-i(1-2c)b_2,$
a direct calculation yields
\[
  (\overline b+2cb)^2=-Db^2.
\]
Using $L=(\overline b+2cb)/D$ and
$\rho=-Db/(\overline b+2cb)$, we obtain
\[
  \rho L=-b,
  \qquad
  \rho^2=-D,
  \qquad
  |\rho|^2=D.
\]
Thus, in every case, substituting these identities and their conjugates
into the left-hand side of \eqref{eq:exponent-match} proves
\eqref{eq:exponent-match}.

\medskip
\noindent\emph{Step 2.}
We prove that
\begin{equation}\label{eq:gram-match}
  \ip{\lambda_mq_m}{\lambda_nq_n}=\ip{h_n}{h_m},
  \qquad m,n\geq0.
\end{equation}
By \eqref{eq:lambda-n}, \eqref{eq:Q-Gram}, and Step~1,
\begin{align*}
 &\sum_{m,n\geq0}
 \frac{\ip{\lambda_mq_m}{\lambda_nq_n}}{m!n!}s^mr^n\\
 &=\frac1{\|q\|^2}
   \ip{Q(\rho s,\cdot)}{Q(\rho\overline r,\cdot)}\\
&=\exp\left\{\frac{c}{D}\bigl(\rho^2s^2+\overline\rho^{\,2}r^2\bigr)
   +\frac{|\rho|^2}{D}sr
   +L\rho s+\overline L\,\overline\rho r\right\}\\
 &=\exp\{sr-sb-r\overline b-c(s^2+r^2)\}.
\end{align*}
On the other hand, \eqref{eq:H-Gram} gives
\[
  \sum_{m,n\geq0}
  \frac{\ip{h_n}{h_m}}{m!n!}s^mr^n
  =\ip{H(r,\cdot)}{H(\overline s,\cdot)}
  =\exp\{sr-sb-r\overline b-c(s^2+r^2)\}.
\]
Comparing the coefficients of $s^mr^n$ proves
\eqref{eq:gram-match}.

By \eqref{eq:biorthogonal}, the two complete families
$\{q_n\}$ and $\{h_n\}$ are biorthogonal with
$\kappa_n=n!$.  Equation~\eqref{eq:gram-match} is exactly the hypothesis
of Lemma~\ref{lem:abstract-gram}, with
$x_n=q_n,y_n=h_n.$
Consequently, \eqref{eq:J-definition} extends uniquely to a conjugation on
$\F$.
\end{proof}

\begin{lemma}\label{J}
Let \(J\) be the conjugation constructed in Lemma~\ref{lem:construction}.
For every \(f\in\mathcal F^{2}\) and \(z\in\C\),
\begin{align}\label{J-integral-representation}
(Jf)(z)
=
\frac{q(z)}{\pi\|q\|}
\int_{\C}
q(w)\e^{\rho zw}\overline{f(w)}
\e^{-|w|^{2}} dA(w).
\end{align}
Moreover, for each fixed \(z\in\C\), the integral above is absolutely
convergent.
\end{lemma}

\begin{proof}
We first determine the action of \(J\) on the reproducing kernels.

Since \(J\) is conjugate-linear and continuous, the generating function
expansion of \(H\) gives
\begin{align}
\bigl(JH(t,\cdot)\bigr)(z)
&=
J\left(
\sum_{n=0}^{\infty}\frac{t^{n}}{n!}h_n
\right)(z) \notag\\
&=
\sum_{n=0}^{\infty}
\frac{\overline t^{\,n}}{n!}(Jh_n)(z) \notag\\
&=
\frac{1}{\|q\|}
\sum_{n=0}^{\infty}
\frac{(\rho\overline t)^{n}}{n!}q_n(z)
\qquad\text{by \eqref{eq:J-definition}} \notag\\
&=
\frac{q(z)}{\|q\|}
\sum_{n=0}^{\infty}
\frac{(\rho\overline t z)^{n}}{n!} \notag\\
&=
\frac{q(z)}{\|q\|}
\e^{\rho\overline t z}.
\label{eq:JH-generating}
\end{align}
Substituting
\(t=\overline w\) into the definition of \(H\), we obtain
\begin{align*}
H(\overline w,z)
=
\exp\left\{
\overline w(z-\overline b)-c\overline w^{\,2}
\right\}
=
\e^{z\overline w}
\e^{-\overline w\,\overline b-c\overline w^{\,2}}
=
K_w(z)
\e^{-\overline w\,\overline b-c\overline w^{\,2}}.
\end{align*}
Consequently,
\begin{equation}\label{eq:kernel-H-relation}
K_w(z)
=
\e^{\overline w\,\overline b+c\overline w^{\,2}}
H(\overline w,z).
\end{equation}

Applying \(J\) to \eqref{eq:kernel-H-relation} and using its
conjugate-linearity, we obtain
\begin{align*}
JK_w
=
J\left(
\e^{\overline w\,\overline b+c\overline w^{\,2}}
H(\overline w,\cdot)
\right)
=
\overline{
\e^{\overline w\,\overline b+c\overline w^{\,2}}
}
\,JH(\overline w,\cdot).
\end{align*}
Since \(c\in\R\),
$
\overline{
\e^{\overline w\,\overline b+c\overline w^{\,2}}
}
=
\e^{wb+cw^{2}}.
$
Setting \(t=\overline w\) in
\eqref{eq:JH-generating} gives
$
\bigl(JH(\overline w,\cdot)\bigr)(z)
=
\frac{q(z)}{\|q\|}
\e^{\rho wz}.$
It follows that
\begin{align*}
(JK_w)(z)
=
\e^{wb+cw^{2}}
\frac{q(z)}{\|q\|}
\e^{\rho wz}
=
\frac{q(z)}{\|q\|}
\exp\{wb+cw^{2}+\rho wz\}.
\end{align*}
Thus,
\begin{equation*}
(JK_w)(z)
=
\frac{q(z)}{\|q\|}
\exp\{wb+cw^{2}+\rho wz\}=\frac{q(z)q(w)}{\|q\|}
\e^{\rho zw}.
\end{equation*}
Fix
\(f\in\mathcal F^{2}\) and \(z\in\C\). By the reproducing property,
\[
(Jf)(z)=\langle Jf,K_z\rangle=\langle JK_z,f\rangle.
\]
Therefore,
\begin{align*}
(Jf)(z)
&=
\langle JK_z,f\rangle\\
&=
\frac{1}{\pi}
\int_{\C}
(JK_z)(w)\overline{f(w)}
\e^{-|w|^{2}}\dd A(w)\\
&=
\frac{1}{\pi}
\int_{\C}
\frac{q(z)q(w)}{\|q\|}
\e^{\rho zw}\overline{f(w)}
\e^{-|w|^{2}}\dd A(w)\\
&=
\frac{q(z)}{\pi\|q\|}
\int_{\C}
q(w)\e^{\rho zw}\overline{f(w)}
\e^{-|w|^{2}}\dd A(w).
\end{align*}
Since \(JK_z\in\mathcal F^{2}\) and
\(f\in\mathcal F^{2}\), the Cauchy--Schwarz inequality yields
\begin{align*}
&\frac{1}{\pi}
\int_{\C}
\left|
(JK_z)(w)\overline{f(w)}
\right|
\e^{-|w|^{2}}\dd A(w)\\
&\qquad\leq
\left(
\frac{1}{\pi}
\int_{\C}
|(JK_z)(w)|^{2}
\e^{-|w|^{2}}\dd A(w)
\right)^{1/2}
\left(
\frac{1}{\pi}
\int_{\C}
|f(w)|^{2}
\e^{-|w|^{2}}\dd A(w)
\right)^{1/2}\\
&\qquad=
\|JK_z\|\,\|f\|<\infty.
\end{align*}
Therefore, the integral is absolutely convergent.
\end{proof}
\begin{theorem}\label{thm:main}
Let \(W_{\psi,\varphi}\) be a bounded weighted composition operator of the
form \eqref{eq:in-reduced-form}. 
Then \(W_{\psi,\varphi}\) is
complex symmetric if and only if either
\begin{equation}
\begin{gathered}
c=0;\\
\text{or}\quad
0<c<\frac12
\quad\text{and}\quad
\frac{(\operatorname{Re}b)^2}{1-2c}
=
\frac{(\operatorname{Im}b)^2}{1+2c}.
\end{gathered}
\end{equation}
In either case, \(W_{\psi,\varphi}\) is complex symmetric with respect
to the conjugation \(J\) defined by
\[
(Jf)(z)
=
\frac{q(z)}{\pi\|q\|}
\int_{\C}
q(w)\e^{\rho zw}\overline{f(w)}
\e^{-|w|^{2}}\dd A(w),
\qquad f\in\F.
\]
\end{theorem}
\begin{proof}
Lemma \ref{lm:main} establishes the necessity.
Let \(J\) be the conjugation constructed in Lemma~\ref{lem:construction}.
For the sufficiency, \eqref{eq:adjoint-eigenvectors} and
\(W_{\psi,\varphi}q_n=a^nq_n\) yield
\[
  JW_{\psi,\varphi}^*h_n
  =J(\overline a^{\,n}h_n)
  =a^n\lambda_nq_n
  =W_{\psi,\varphi}Jh_n.
\]
Both operators are bounded and the \(h_n\)'s have dense linear span.
Therefore \(JW_{\psi,\varphi}^*=W_{\psi,\varphi}J\) on \(\F\).
Thus the sufficiency follows from Lemma~\ref{lem:construction}, and
Lemma~\ref{J} gives the explicit formula for \(J\).
\end{proof}
\begin{corollary}\label{cor:wcg}
Let \(J\) be the conjugation constructed in \eqref{J-integral-representation}.
If \(c=0\), then 
\begin{align}\label{antilinear}
(Jf)(z)
=
\e^{bz-|b|^{2}/2}
\overline{
f\bigl(\overline{b+\rho z}\bigr)
}\qquad f\in\mathcal F^{2}.
\end{align}
\end{corollary}
\begin{proof}
When \(c=0\), we have \(q(z)=\e^{bz}\), and
\eqref{eq:moments} yields
$\|q\|=\e^{|b|^{2}/2}.$
Hence \eqref{J-integral-representation} reduces to
\begin{equation}\label{eq:J-c-zero-integral}
(Jf)(z)
=
\frac{\e^{bz-|b|^{2}/2}}{\pi}
\int_{\C}
\e^{(b+\rho z)w}\overline{f(w)}
\e^{-|w|^{2}}\dd A(w).
\end{equation}
Since
$
\e^{(b+\rho z)w}
=
K_{\overline{b+\rho z}}(w),
$
the reproducing property gives
\begin{align*}
\frac{1}{\pi}
\int_{\C}
\e^{(b+\rho z)w}\overline{f(w)}
\e^{-|w|^{2}}\dd A(w)
=
\left\langle
K_{\overline{b+\rho z}},f
\right\rangle
=
\overline{f\bigl(\overline{b+\rho z}\bigr)}.
\end{align*}
Substituting this identity into \eqref{eq:J-c-zero-integral}, we obtain
\begin{equation*}
(Jf)(z)
=
\e^{bz-|b|^{2}/2}
\overline{f\bigl(\overline{b+\rho z}\bigr)}.
\end{equation*}
\end{proof}

\begin{remark}\label{re:hai}
If $c=0$, then 
\begin{equation*}
  \rho=
  \begin{cases}
    -b/\bar{b}
    \,
      & b\ne0,\\[3mm]
    1,
      & b=0.
  \end{cases}
\end{equation*}
In this case, the conjugation $J$ defined in \eqref{antilinear} is a
weighted composition conjugation, as characterized in
\cite[Theorem~3.1]{HaiKhoi2016} and
\cite[Theorem~3.13]{HaiKhoi2018}.
In particular, when $b=c=0$, $J$ reduces to the standard conjugation
\begin{align}\label{K}
\mathcal Kf(z)=\overline{f(\overline z)}.
\end{align}

In contrast, when $c>0$, the conjugation $J$ defined in
\eqref{J-integral-representation} is not a weighted composition conjugation. Indeed,
the results in
\cite{HaiKhoi2016,HaiKhoi2018} show that the weight of every such
conjugation on $\F$ must be of the form $\kappa e^{\tau z}$,
whereas
$J1=\frac{q}{\|q\|}
  =\frac{1}{\|q\|}\exp(cz^2+bz)$
has a nonzero quadratic term in its exponent.
\end{remark}

We next make the connection
with canonical integral operators precise in the centered case.

\begin{proposition}\label{prop:canonical-relation}
Assume $b=0$.  Put
\begin{align}\label{st}
s=\frac1\rho,
  \qquad
  t=\frac{2c}{\rho}.
\end{align}
Let $T^{(s,t)}$ be the canonical integral operator of Dong and Zhu \cite{DongZhu2024},
\[
  (T^{(s,t)}f)(z)
  =\frac1{\pi\sqrt s }\int_{\mathbb C}
  \exp\!\left(
    \frac{tz^2-\overline t\,\overline w^{\,2}+2z\overline w}{2s}
  \right)f(w)e^{-|w|^2}\,dA(w),
\]
where either of the two square roots of \(s\) may be used in the prefactor.
Then
\[
  J\mathcal K=\omega T^{(s,t)},
  \qquad
  \omega=\sqrt s/\|q\|,
  \qquad |\omega|=1.
\]
Moreover, $|s|^2=|t|^2+1$, so $T^{(s,t)}$ is unitary.
Changing the choice of \(\sqrt{s}\) changes both \(T^{(s,t)}\) and
\(\omega\) by a sign; it therefore does not affect the identity or the
unitarity conclusion.
\end{proposition}

\begin{proof}
When $b=0$, changing variables $w\mapsto\overline w$ gives
\[
  (J\mathcal Kf)(z)
  =\frac1{\|q\|}\int_{\mathbb C}
  \exp\{cz^2+c\overline w^{\,2}+\rho z\overline w\}
  f(w)e^{-|w|^2}\,\frac{dA(w)}{\pi}.
\]
If $c>0$, then \eqref{eq:rho} gives \(\rho=i\sqrt D\), while
\eqref{st} implies that
$
  \frac{t}{2s}=c,
  -\frac{\overline t}{2s}=c.
$
Hence 
\[
cz^2+c\overline w^{\,2}+\rho z\overline w
=\frac{tz^2-\overline t\,\overline w^{\,2}+2z\overline w}{2s}.
\]
The same identities are immediate when $c=0$.  Thus the two kernels agree
up to the scalar $\omega$.  By Lemma \ref{lem:moments}, we have
$\|q\|^2=\frac1{\sqrt D}$. Since $|s|=\frac1{\sqrt D}$,
$|\omega|=\frac{\sqrt{|s|}}{\|q\|}=1$. Equation~\eqref{ABD} implies that
$|s|^2-|t|^2
  =\frac{1-4c^2}{D}=1$. Theorem~A of \cite{DongZhu2024} now shows that
\(T^{(s,t)}\) is unitary.
\end{proof}
The following three examples exhibit complex symmetric but nonnormal
weighted composition operators. The first two are compact and correspond,
respectively, to the cases \(c=0\) and \(0<c<\frac12\) in
Theorem~\ref{thm:main}, whereas the third is noncompact with compact
square and corresponds to the latter case.
\begin{example}
Let
$a=\frac12,\ c=0,\ b=1,$
and set \(u=0\) and \(v=b(1-a)=\frac12\). Then
\begin{align*}
\varphi(z)=\frac12 z,\quad \psi(z)=\exp\left(\frac12 z\right).
\end{align*}
Since $2|u|=0<1-(\frac12)^2,$
Proposition~\ref{prop:explicit-boundedness}\textup{(i)}  implies that
$W_{\psi,\varphi}$ is compact. Since $c=0$, Theorem~\ref{thm:main}
shows that $W_{\psi,\varphi}$ is complex symmetric.
Finally, Proposition~\ref{normal} shows that
$W_{\psi,\varphi}$ is not normal.
\end{example}

\begin{example}
Let
$a=\frac12,\ c=\frac14,\ b=1+i\sqrt3,$
and set \(u=c(1-a^2)=\frac3{16}\) and
\(v=b(1-a)=\frac12+i\frac{\sqrt3}{2}\). Then
\begin{align*}
\varphi(z)=\frac12 z,\quad \psi(z)=\exp\left(\frac3{16}z^2+
  \left(\frac12+i\frac{\sqrt3}{2}\right)z\right).
\end{align*}
Since $2|u|=2\cdot (3/16)<1-1/4,$
Proposition~\ref{prop:explicit-boundedness}\textup{(i)}  implies that
$W_{\psi,\varphi}$ is compact. Moreover,
$
\frac{(\operatorname{Re}b)^2}{1-2c}
=\frac{(\operatorname{Im}b)^2}{1+2c}
=2,
$
so Theorem~\ref{thm:main} shows that
$W_{\psi,\varphi}$ is complex symmetric.
Finally, Proposition~\ref{normal} shows that
$W_{\psi,\varphi}$ is not normal.
\end{example}

\begin{example}
Let $a=\frac{i}{2}, c=\frac{3}{10}, b=1-2i,$
and set
$
  u=c(1-a^2)=\frac38,
  \
  v=b(1-a)=-\frac52 i.
$
Then
\[
  \varphi(z)=\frac{i}{2}z,
  \qquad
  \psi(z)=
  \exp\left(\frac38z^2-\frac52iz\right).
\]

A direct computation gives
$
  2|u|
  =\frac34
  =1-|a|^2,
$
and
$
|u|v+u\overline v
=\frac38\left(-\frac52i\right)
+\frac38\left(\frac52i\right)
=0.
$
Thus, the critical boundedness condition in
Proposition~\ref{prop:explicit-boundedness}\textup{(ii)} is satisfied.
Consequently, \(W_{\psi,\varphi}\) is bounded but not compact.

Moreover, the composition rule for weighted composition operators gives
$
  W_{\psi,\varphi}^2
  =
  W_{\psi_2,\varphi_2},
$
where
$
  \varphi_2(z)=a^2z=-\frac14z
$
and
$
  \psi_2(z)
  =\psi(z)\psi(az)
  =
  \exp(
    \frac{9}{32}z^2+
    \left(\frac54-\frac52i\right)z
  ).
$
Since
$
  2\cdot\frac{9}{32}
  =
  \frac{9}{16}
  <
  \frac{15}{16}
  =
  1-|a^2|^2,
$
Proposition~\ref{prop:explicit-boundedness}\textup{(i)} implies that
\(W_{\psi,\varphi}^2\) is compact.

Furthermore,
$
  0<c=\frac{3}{10}<\frac12
$
and
$
  \frac{(\operatorname{Re}b)^2}{1-2c}
  =
  \frac{(\operatorname{Im}b)^2}{1+2c}
  =\frac52.
$
It follows from Theorem~\ref{thm:main} that
\(W_{\psi,\varphi}\) is complex symmetric.

Finally, \(W_{\psi,\varphi}\) is not normal. Indeed, since
\(K_0\equiv1\), formula~\eqref{eq:adjoint-kernel} gives
$
W_{\psi,\varphi}^*K_0
=\overline{\psi(0)}K_{\varphi(0)}
=K_0.
$
On the other hand,
$
W_{\psi,\varphi}K_0=\psi.
$
Since
\[
  \psi(z)
  =
  \exp\left(\frac38z^2-\frac52iz\right)
  =
  1-\frac52iz+\cdots
\]
and the monomials are mutually orthogonal in \(\mathcal F^2\), we have
\[
\begin{aligned}
  \left\|W_{\psi,\varphi}K_0\right\|^2
  =\|\psi\|^2
  \geq
  1+\left|-\frac52i\right|^2
  =\frac{29}{4}
  >1
  =\left\|W_{\psi,\varphi}^*K_0\right\|^2.
\end{aligned}
\]
It follows that
\(W_{\psi,\varphi}\) is not normal.

Thus, \(W_{\psi,\varphi}\) is a nonnormal, noncompact, complex symmetric
weighted composition operator whose square is compact.
\end{example}

\section{Complete classification}

\begin{theorem}\label{thm:weighted-affine}
Fix $\alpha>0$. Let $\Psi$ and $\Phi$ be entire functions
with $\Psi\not\equiv0$, and suppose that
$W_{\Psi,\Phi}^{(\alpha)}$ is bounded on
$\mathcal F_\alpha^2$, where
\[
  \Phi(z)=az+d,\qquad |a|\leq1.
\]
Then $W_{\Psi,\Phi}^{(\alpha)}$ is complex symmetric if
and only if one of the following conditions holds.

\begin{enumerate}
\item[\textup{(I)}]
$a=0$. In this case,
\[
  W_{\Psi,\Phi}^{(\alpha)}
  =
  \Psi\otimes K_d^{(\alpha)},
\]
and hence $W_{\Psi,\Phi}^{(\alpha)}$ has rank one.

\item[\textup{(II)}]
$|a|=1$ and
\[
  \Psi(z)
  =
  \mu e^{-\alpha a\overline d\,z}
\]
for some $\mu\in\mathbb C\setminus\{0\}$. In this case,
$W_{\Psi,\Phi}^{(\alpha)}$ is normal.

\item[\textup{(III)}]
$0<|a|<1$ and
\[
  \Psi(z)
  =
  \exp\bigl(\beta z^2+\gamma z+\delta\bigr),
  \qquad z\in\mathbb C,
\]
for some \(\beta,\gamma,\delta\in\mathbb C\) with
\( |\beta|<\frac{\alpha}{2}\).
Set
$p=\frac{d}{1-a}.$
If $\beta=0$, take $\eta=1$. If $\beta\neq0$, choose a unimodular
constant $\eta$ such that
$
\frac{\eta^2\beta}{\alpha(1-a^2)}
=|\frac{\beta}{\alpha(1-a^2)}|.
$
Define
\[
  c=\frac{\eta^2\beta}{\alpha(1-a^2)},
  \qquad
  b=\eta\left(
    \frac{2\beta p+\gamma}{1-a}-\alpha\overline p
  \right).
\]
Then either
\begin{enumerate}
\item[\textup{(i)}]
\(\beta=0\); or

\item[\textup{(ii)}]
$0<|\beta|<
\frac{\alpha}{2}|1-a^2|$ and
\begin{equation}\label{eq:weighted-main-condition}
  \frac{(\operatorname{Re}b)^2}{1-2c}
  =
  \frac{(\operatorname{Im}b)^2}{1+2c}.
\end{equation}
\end{enumerate}
\end{enumerate}

In case \textup{(III)}, set
$
D=1-4c^2,\
L=\frac{\overline b+2cb}{D},\
q_\alpha(z)=\exp\bigl(\alpha c z^2+bz\bigr),
$
so that \(D>0\), and \(L\ne0\) whenever \(b\ne0\). Define
\begin{equation}\label{eq:weighted-rho}
  \rho=
  \begin{cases}
    -\dfrac{b}{L},
      & b\ne0,\\[3mm]
    1,
      & b=0,\ c=0,\\[1mm]
    i\sqrt D,
      & b=0,\ c>0.
  \end{cases}
\end{equation}
Then \(q_\alpha\in\mathcal F_\alpha^2\). For
$f\in\mathcal F_\alpha^2,$
the conjugation implementing the
complex symmetry of
$W_{\Psi,\Phi}^{(\alpha)}$ is given by
\begin{equation}\label{C-alpha}
\begin{aligned}
(C_\alpha f)(z)
={}&
\frac{
  \alpha
  e^{\alpha\overline pz-\alpha|p|^2}
  q_\alpha\bigl(\overline\eta(z-p)\bigr)
}{
  \pi\|q_\alpha\|_\alpha
}
\\
&\times
\int_{\mathbb C}
q_\alpha(w)
\exp\!\left\{
  \alpha\rho\overline\eta(z-p)w
  -\alpha p\overline\eta\,\overline w
\right\}
\overline{f(\eta w+p)}
e^{-\alpha|w|^2}\,\dd A(w).
\end{aligned}
\end{equation}
\end{theorem}

\begin{proof}
Let
\[
  D_\alpha:\mathcal F_\alpha^2\longrightarrow\F,
  \qquad
  (D_\alpha f)(z)=f\left(\frac z{\sqrt\alpha}\right).
\]
Set
\begin{align}\label{az}
\widetilde\psi(z)=\Psi\left(\frac z{\sqrt\alpha}\right),\quad
\widetilde\varphi(z)=az+\sqrt\alpha\,d. 
\end{align}
As observed in Section~2, $D_\alpha$ is unitary and
\begin{equation}\label{eq:weighted-affine-dilation}
 D_\alpha W_{\Psi,\Phi}^{(\alpha)}D_\alpha^{-1}
  =W_{\widetilde\psi,\widetilde\varphi}.
\end{equation}
Thus $W_{\Psi,\Phi}^{(\alpha)}$ is complex symmetric if and only if
$W_{\widetilde\psi,\widetilde\varphi}$ is complex symmetric.

If $a=0$, then
boundedness gives
\(\Psi=W_{\Psi,\Phi}^{(\alpha)}1\in\mathcal F_\alpha^2\).
Lemma~\ref{rank1} therefore gives
\[
  W_{\Psi,\Phi}^{(\alpha)}
  =\Psi\otimes K_d^{(\alpha)},
\]
which is complex symmetric. This proves \textup{(I)}.

Suppose that $|a|=1$. Proposition~\ref{prop:boundedness}, applied to
$W_{\widetilde\psi,\widetilde\varphi}$, gives
\begin{align}\label{mu}
  \widetilde\psi(z)=\mu e^{-a\sqrt{\alpha}\,\overline d\,z}
  \quad\text{and}\quad
  \Psi(z)=\mu e^{-\alpha a\overline d\,z}
\end{align}
for some \(\mu\ne0\).
Conversely, this formula and Proposition~\ref{normal} show that
$W_{\widetilde\psi,\widetilde\varphi}$ is normal. By unitary equivalence,
$W_{\Psi,\Phi}^{(\alpha)}$ is also normal and hence complex symmetric. Thus
\textup{(II)} is necessary and sufficient.

Suppose that $0<|a|<1$. First assume that
$W_{\Psi,\Phi}^{(\alpha)}$ is complex
symmetric. Then $W_{\widetilde\psi,\widetilde\varphi}$ is complex symmetric, so
Proposition~\ref{prop:cs-symbol-form} gives
\[
  \widetilde\psi(z)
  =\exp\bigl(\widetilde\beta z^2
  +\widetilde\gamma z+\delta\bigr),
  \qquad
  |\widetilde\beta|<\frac12,
\]
for some $\widetilde\beta,\widetilde\gamma,\delta\in\mathbb C$.
Define
$
  \beta:=\alpha\widetilde\beta,
  \
  \gamma:=\sqrt\alpha\,\widetilde\gamma.
$
Then
\[
  \Psi(z)=\exp\bigl(\beta z^2+\gamma z+\delta\bigr),
  \qquad |\beta|<\frac\alpha2.
\]

We now reduce every bounded operator with a weight of this form. Put
\[
  \widetilde\beta=\frac{\beta}{\alpha},
  \qquad
  \widetilde\gamma=\frac{\gamma}{\sqrt\alpha},
  \qquad
  p=\frac{d}{1-a},
  \qquad
  \widetilde p=\sqrt\alpha\,p,
\]
and
\[
  \psi_2(z)=\exp\bigl(\widetilde\beta z^2
    +\widetilde\gamma z\bigr).
\]
The choice of $\eta$ in the statement is precisely the choice required in
Lemma~\ref{lem:reduction}, since $\alpha>0$. Applying that lemma to
$\psi_2$ and using $\widetilde\psi=e^\delta\psi_2$ gives
\begin{equation}\label{eq:weighted-affine-reduction}
  W_{\widetilde\psi,\widetilde\varphi}
  =\widetilde\psi(\widetilde p)\,
   V W_{\psi_0,\varphi_0}V^*,
  \qquad
  V=U_{\widetilde p}R_{\overline\eta},
\end{equation}
where
\[
  \varphi_0(z)=az,
  \qquad
  \psi_0(z)=\exp(uz^2+vz),
\]
with
\[
  u=\eta^2\widetilde\beta,
  \qquad
  v=\eta\left(
    2\widetilde\beta\widetilde p+\widetilde\gamma
    +(a-1)\overline{\widetilde p}
  \right).
\]
Here \(V\) is unitary and
\[
  \widetilde\psi(\widetilde p)=\Psi(p)\ne0.
\]
Thus \eqref{eq:weighted-affine-dilation},
\eqref{eq:weighted-affine-reduction}, and Lemma~\ref{nozero} show that
$W_{\Psi,\Phi}^{(\alpha)}$ is complex symmetric if and only if
$W_{\psi_0,\varphi_0}$ is complex symmetric. The reduced parameters are
\[
  \widetilde{c}
  =\frac{u}{1-a^2}
  =\frac{\eta^2\widetilde\beta}{1-a^2}=c,
  \qquad
  \widetilde b
  =\frac{v}{1-a}
  =\eta\left(
    \frac{2\widetilde\beta\widetilde p+\widetilde\gamma}{1-a}
    -\overline{\widetilde p}
  \right)
  =\frac{b}{\sqrt\alpha}.
\]
Since $\sqrt\alpha>0$, condition \eqref{eq:weighted-main-condition} for
$b$ is equivalent to the corresponding condition in
Theorem~\ref{thm:main} for $\widetilde b$.

By the choice of the unimodular constant \(\eta\),
\[
c
=
\frac{\eta^2\beta}{\alpha(1-a^2)}
=
\left|
\frac{\beta}{\alpha(1-a^2)}
\right|
=
\frac{|\beta|}{\alpha|1-a^2|}.
\]
The same identity remains valid when \(\beta=0\), in which case
\(\eta=1\). Consequently, \(c=0\) if and only if \(\beta=0\).
Moreover, when \(\beta\ne0\), the condition
$0<c<\frac12$
is equivalent to
$0<|\beta|<\frac{\alpha}{2}|1-a^2|.$

Therefore,
Lemma~\ref{lem:reduced-c-range}
and the necessity in Theorem~\ref{thm:main} now give precisely
\textup{(III)(i)} or \textup{(III)(ii)}. 

Conversely, either of those
conditions puts $W_{\psi_0,\varphi_0}$ under the sufficiency part of
Theorem~\ref{thm:main}; hence
\eqref{eq:weighted-affine-dilation} and
\eqref{eq:weighted-affine-reduction} imply that
$W_{\Psi,\Phi}^{(\alpha)}$ is complex symmetric. 

It remains to verify the formula for the implementing conjugation in
case \textup{(III)}. In either subcase, \(0\le c<1/2\); hence
\(q_\alpha\in\mathcal F_\alpha^2\). Define
$
  q(z)=\exp\bigl(cz^2+\widetilde b z\bigr).
$
Moreover,
$
  \widetilde L
  =\frac{\overline{\widetilde b}+2c\widetilde b}{D}
  =\frac{L}{\sqrt\alpha}.
$
Thus \eqref{eq:weighted-rho} gives exactly the same value of $\rho$ as
\eqref{eq:rho} applied to the reduced parameter $\widetilde b$,
including all three cases. Let $J$ be the
conjugation from Theorem~\ref{thm:main} for 
$W_{\psi_0,\varphi_0}$, and set
$
  \widetilde C
  =U_{\widetilde p}R_{\overline\eta}JR_\eta U_{-\widetilde p},
  \
  C_\alpha=D_\alpha^{-1}\widetilde C D_\alpha.
$
By \eqref{eq:weighted-affine-reduction} and
\eqref{eq:weighted-affine-dilation}, $C_\alpha$ implements the complex symmetry
of $W_{\Psi,\Phi}^{(\alpha)}$.

For $g\in\F$, the integral representation
\eqref{J-integral-representation} gives
\begin{equation}\label{for:C}
\begin{aligned}
(\widetilde Cg)(z)
={}&
\frac{
 e^{\overline{\widetilde p}z-|\widetilde p|^2}
 q\bigl(\overline\eta(z-\widetilde p)\bigr)
}{\pi\|q\|}
\\
&\times
\int_{\mathbb C}
q(w)
\exp\!\left\{
  \rho\overline\eta(z-\widetilde p)w
  -\widetilde p\,\overline\eta\,\overline w
\right\}
\overline{g(\eta w+\widetilde p)}e^{-|w|^2}\,\dd A(w).
\end{aligned}
\end{equation}
Indeed, this follows by substituting
$
  \overline{(R_\eta U_{-\widetilde p}g)(w)}
  =\exp\left(
    -\widetilde p\,\overline\eta\,\overline w
    -\frac{|\widetilde p|^2}{2}
  \right)
  \overline{g(\eta w+\widetilde p)}
$
into \eqref{J-integral-representation} and combining the two factors
$e^{-|\widetilde p|^2/2}$.

We now compute the explicit formula for \(C_\alpha\).
Note that
$
q(z)
=\exp\!\left(
cz^2+\frac{b}{\sqrt{\alpha}}z
\right)
=q_\alpha\!\left(\frac{z}{\sqrt{\alpha}}\right),
q(\sqrt{\alpha}\,w)
=q_\alpha(w).
$
The change of
variables \(z=\sqrt{\alpha}\,w\) gives
\[
\begin{aligned}
\|q\|^2
&=\frac{1}{\pi}\int_{\mathbb C}
|q(z)|^2e^{-|z|^2}\,\dd A(z)\\
&=\frac{\alpha}{\pi}\int_{\mathbb C}
|q(\sqrt{\alpha}\,w)|^2e^{-\alpha|w|^2}\,\dd A(w)\\
&=\frac{\alpha}{\pi}\int_{\mathbb C}
|q_\alpha(w)|^2e^{-\alpha|w|^2}\,\dd A(w)\\
&=\|q_\alpha\|_\alpha^2.
\end{aligned}
\]
To expand
$
C_\alpha=D_\alpha^{-1}\widetilde C D_\alpha,
$
note that
$(D_\alpha^{-1}h)(z)=h(\sqrt{\alpha}\,z),$
the preceding integral formula \eqref{for:C} for \(\widetilde C\) yields
\[
\begin{aligned}
(C_\alpha f)(z)
&=(D_\alpha^{-1}\widetilde C D_\alpha f)(z)\\
&=\bigl[\widetilde C(D_\alpha f)\bigr](\sqrt{\alpha}\,z)\\
&={}
\frac{
e^{\overline{\widetilde p}\sqrt{\alpha}z-|\widetilde p|^2}
q\bigl(\overline{\eta}(\sqrt{\alpha}z-\widetilde p)\bigr)
}{
\pi\|q\|
}\\
&\quad\times
\int_{\mathbb C}q(w)
\exp\!\left\{
\rho\overline{\eta}(\sqrt{\alpha}z-\widetilde p)w
-\widetilde p\,\overline{\eta}\,\overline{w}
\right\}\\
&\hspace{3.2cm}\times
\overline{
(D_\alpha f)(\eta w+\widetilde p)
}
e^{-|w|^2}\,\dd A(w).
\end{aligned}
\]
Since $
\widetilde p=\sqrt{\alpha} p,\
\widetilde b=\frac{b}{\sqrt{\alpha}},$
the factors outside the integral satisfy
\[
\begin{aligned}
\overline{\widetilde p}\sqrt{\alpha}z-|\widetilde p|^2
&=\alpha\overline pz-\alpha|p|^2,\\
q\bigl(\overline{\eta}(\sqrt{\alpha}z-\widetilde p)\bigr)
&=q\bigl(\sqrt{\alpha}\,\overline{\eta}(z-p)\bigr)
=q_\alpha\bigl(\overline{\eta}(z-p)\bigr).
\end{aligned}
\]
Moreover,
\[
\begin{aligned}
(D_\alpha f)(\eta w+\widetilde p)
&=
f\!\left(
\frac{\eta w+\widetilde p}{\sqrt{\alpha}}
\right)
=
f\!\left(
\frac{\eta w}{\sqrt{\alpha}}+p
\right),\\
\rho\overline{\eta}(\sqrt{\alpha}z-\widetilde p)w
&=
\sqrt{\alpha}\,
\rho\overline{\eta}(z-p)w,\\
-\widetilde p\,\overline{\eta}\,\overline{w}
&=
-\sqrt{\alpha}\,
p\overline{\eta}\,\overline{w}.
\end{aligned}
\]
Substitution of these identities gives 
\[
\begin{aligned}
(C_\alpha f)(z)
={}&
\frac{
e^{\alpha\overline pz-\alpha|p|^2}
q_\alpha\bigl(\overline{\eta}(z-p)\bigr)
}{
\pi\|q_\alpha\|_\alpha
}\\
&\times
\int_{\mathbb C}q(w)
\exp\!\left\{
\sqrt{\alpha}\,
\rho\overline{\eta}(z-p)w
-\sqrt{\alpha}\,
p\overline{\eta}\,\overline{w}
\right\}\\
&\hspace{3.2cm}\times
\overline{
f\!\left(
\frac{\eta w}{\sqrt{\alpha}}+p
\right)
}
e^{-|w|^2}\,\dd A(w).
\end{aligned}
\]
Make the change of variables
$
w =\sqrt{\alpha}\lambda
$
in the integral. 
Since
$\dd A(w)=\alpha\dd A(\lambda),$
\eqref{C-alpha} follows.
\end{proof}
\begin{remark}
When \(\alpha=1\), it follows from
\cite[Theorem~3.15]{HaiKhoi2016} and Remark~\ref{re:hai}
that the operators described in Case~\textup{(II)} and
Case~\textup{(III)(i)} of Theorem~\ref{thm:weighted-affine}
are complex symmetric with respect to weighted composition
conjugations of the form \eqref{f:C}.
\end{remark}

Hai and Khoi \cite[Theorem 3.18]{HaiKhoi2016} determined the spectra for
operators symmetric with respect to
\(\mathcal{C}_{\gamma,\tau,\kappa}\) \eqref{f:C}.
In fact, except for the noncompact case with $|a|<1$, the spectra of the 
operators in all the other cases have already been completely characterized 
in the literature; see, for example, \cite{GuoIzuchi2008,Zhao2014}.

We shall use the following consequence of
\cite[Corollary~1.4]{Zhao2014}, obtained via the unitary operator
\[
  D^{-1}_{1/2}:\F\longrightarrow\mathcal F_{1/2}^{2},
  \qquad
  (D^{-1}_{1/2}f)(z)=f\left(\frac z{\sqrt2}\right).
\]
\begin{lemma}\label{zhao}
Suppose that \(a,b\in\mathbb C\) satisfy
$|a|=1
\ \text{and} \
b\ne0.$
Set
$\varphi(z)=az+b$ and
$\psi(z)=k_{-\overline a b}(z)
=\exp(
-a\overline b\,z-\frac{|b|^2}{2}).$
Then \(W_{\psi,\varphi}\) is a unitary weighted composition
operator on $\F$, and
\[
\sigma(W_{\psi,\varphi})
=
\begin{cases}
\displaystyle
\overline{\{\beta a^m:m\geq 0\}},
& a\ne1,\\[2mm]
\mathbb T,
& a=1,
\end{cases}
\]
where, for \(a\ne1\),
$
\beta
=
\exp\left(
\frac{|b|^2}{2}\frac{a+1}{a-1}
\right).
$
Equivalently, if
$
p=\frac{b}{1-a}
$, then
$\beta=\psi(p).$
Here, \(\mathbb T\) denotes the unit circle.
\end{lemma}
\begin{proposition}\label{pu}
Fix $\alpha>0$, and let $W_{\Psi,\Phi}^{(\alpha)}$ be a bounded complex
symmetric weighted composition operator characterized in
Theorem~\ref{thm:weighted-affine}, with
\(\Phi(z)=az+d\). 
Whenever \(a\ne1\), set \(p=d/(1-a)\).
Then
\[
\sigma\bigl(W_{\Psi,\Phi}^{(\alpha)}\bigr)
=
\begin{cases}
\{0,\Psi(d)\},
& a=0,\\[1ex]

\Psi(0) e^{\alpha|d|^{2}/2}\mathbb T,
& a=1,\ d\neq0,\\[1ex]

\{\Psi(0)\},
& a=1,\ d=0,\\[1ex]

\overline{\{\Psi(p)a^n:n\ge0\}},
& |a|=1,\ a\neq1,\\[1ex]

\{0\}\cup\{\Psi(p)a^n:n\ge0\},
&0<|a|<1,
\end{cases}
\]
where $\mathbb T$ denotes the unit circle.
\end{proposition}

\begin{proof}
By \eqref{eq:weighted-affine-dilation},
$\sigma\bigl(W_{\Psi,\Phi}^{(\alpha)}\bigr)
  =\sigma\bigl(W_{\widetilde\psi,\widetilde\varphi}\bigr).$

We consider the cases in Theorem~\ref{thm:weighted-affine} separately.

\begin{enumerate}[
  label=\textup{(\roman*)},
  leftmargin=*,
  align=left,
  itemsep=0.3em,
  topsep=0.3em,
  parsep=0pt
]
\item
Suppose that $a=0$. Then
\[
  W_{\Psi,\Phi}^{(\alpha)}
  =\Psi\otimes K_d^{(\alpha)},
  \qquad
  \bigl(W_{\Psi,\Phi}^{(\alpha)}\bigr)^2
  =\Psi(d)W_{\Psi,\Phi}^{(\alpha)}.
\]
Therefore,
$\sigma(W_{\Psi,\Phi}^{(\alpha)})
\subset\{0,\Psi(d)\}$.

The orthogonal complement of \(K_d^{(\alpha)}\) is contained in the
kernel of $W_{\Psi,\Phi}^{(\alpha)}$, so
$0\in\sigma(W_{\Psi,\Phi}^{(\alpha)})$.
If $\Psi(d)\neq0$, then
$W_{\Psi,\Phi}^{(\alpha)}\Psi=\Psi(d)\Psi$, so
$\Psi(d)\in\sigma(W_{\Psi,\Phi}^{(\alpha)})$.
If $\Psi(d)=0$, the displayed set is simply \(\{0\}\).
Consequently,
\[
  \sigma\bigl(W_{\Psi,\Phi}^{(\alpha)}\bigr)
  =\{0,\Psi(d)\}.
\]

\item
Suppose that $|a|=1$.
If $a=1$ and $d=0$, \eqref{az} and \eqref{mu} imply
$W_{\widetilde\psi,\widetilde\varphi}=\mu I$ and $\mu=\Psi(0),$ so 
\(\sigma(W_{\widetilde\psi,\widetilde\varphi})=\{\Psi(0)\}\).

If $a\ne1$ and $d=0$, then
$W_{\widetilde\psi,\widetilde\varphi}$ is diagonal with respect to the standard orthonormal
basis $e_n(z)=z^n/\sqrt{n!}$, since
$W_{\widetilde\psi,\widetilde\varphi}e_n=\Psi(0) a^n e_n$. Therefore,
\[
  \sigma(W_{\widetilde\psi,\widetilde\varphi})
  =\overline{\{\Psi(0)a^n:n\geq0\}}.
\]

Assume that $d\neq0$. By \eqref{mu},
$\widetilde\psi(z)=\mu e^{-a\sqrt\alpha\bar d z}$.
Define
\[
  \tilde{\psi_1}(z)
  =
  \exp\left(
    -a\sqrt\alpha\bar d z-\frac{\alpha|d|^2}{2}
  \right)=k_{-\bar{a}\sqrt\alpha d}(z).
\]
Then
$W_{\widetilde\psi,\widetilde\varphi}
=\mu e^{\alpha|d|^2/2}
 W_{\widetilde\psi_1,\widetilde\varphi}$.
Using Lemma \ref{zhao}, we have
\[
\sigma(W_{\widetilde\psi_1,\widetilde\varphi})
=
\begin{cases}
\mathbb T,
& a=1,d\neq 0,\\[1ex]
\overline{\{\omega a^n:n\ge0\}},
& a\neq1, d\neq 0,
\end{cases}
\]
where
$\omega
  =
  \exp\{
    \frac{\alpha|d|^2}{2}\frac{a+1}{a-1}\}.$
When $a\neq1$, we have
$
  \mu e^{\alpha|d|^2/2}\omega
  =
  \mu\exp\left(\frac{a\alpha|d|^2}{a-1}\right)
  =
  \Psi(p).
$
Therefore,
\[
  \sigma(W_{\widetilde\psi,\widetilde\varphi})
  =
  \begin{cases}
    \mu e^{\alpha|d|^{2}/2}\mathbb T,
      &a=1,\ d\neq0,\\[1ex]
    \{\Psi(0)\},
      &a=1,\ d=0,\\[1ex]
    \overline{\{\Psi(p)a^n:n\ge0\}},
      &|a|=1,\ a\neq1.
  \end{cases}
\]
\item
Suppose that $0<|a|<1$, and \eqref{eq:weighted-affine-reduction} gives
$
  W_{\widetilde\psi,\widetilde\varphi}
  =\Psi(p)\,V W_{\psi_0,\varphi_0}V^*.
$
By Corollary~\ref{spectrum}, we have
$\sigma(W_{\psi_0,\varphi_0})
  =\{0\}\cup\{a^n:n\ge0\}.$
Since \(\Psi(p)\ne0\), scalar multiplication and unitary equivalence give
\[
  \sigma\bigl(W_{\widetilde\psi,\widetilde\varphi}\bigr)
  =\{0\}\cup\{\Psi(p)a^n:n\ge0\}.
\]
\end{enumerate}
\end{proof}

\end{document}